\documentclass[english,12pt,oneside]{amsproc}
\usepackage[english]{babel}
\usepackage{a4wide}
\usepackage{amsthm}
\usepackage{graphics}
\usepackage{amsfonts, amssymb, amscd, amsmath}
\usepackage{latexsym}
\usepackage[matrix,arrow,curve]{xy}
\usepackage{mathabx}
\usepackage{color}
\usepackage{pbox}
\usepackage{tikz}
\usetikzlibrary{matrix,decorations.pathreplacing,positioning}
\usepackage{hyperref}

 \DeclareMathOperator{\id}{id}

\DeclareMathOperator{\const}{const}
\DeclareMathOperator{\Hom}{Hom} 
 
\DeclareMathOperator{\diag}{diag}\DeclareMathOperator{\conv}{conv}
\DeclareMathOperator{\Pe}{Pe} \DeclareMathOperator{\Sq}{Sq}
 \DeclareMathOperator{\Fl}{Fl}
 
\DeclareMathOperator{\St}{St} \DeclareMathOperator{\link}{link}
\DeclareMathOperator{\Hilb}{Hilb}
\DeclareMathOperator{\Spec}{Spec}

\newcommand{\Zo}{\mathbb{Z}}
\newcommand{\Ro}{\mathbb{R}}

\newcommand{\Co}{\mathbb{C}}
\newcommand{\Qo}{\mathbb{Q}}

\newcommand{\Zt}{\Zo_2}

\newcommand{\br}{\widetilde{\beta}}

\newcommand{\ca}[1]{\mathcal{#1}}
\newcommand{\Hr}{\widetilde{H}}
\newcommand{\dd}{\partial}
\newcommand{\Ca}{\mathcal{C}}
\newcommand{\F}{\mathcal{F}}
\newcommand{\B}{\mathcal{B}}
\newcommand{\I}{\mathbb{I}}

\newcommand{\Xs}{X^6_{\St}}
\newcommand{\Qs}{Q^3_{\St}}

\newcommand{\Nt}{\mathcal{N}}

\newcounter{stmcounter}[section]

\numberwithin{equation}{section}

\theoremstyle{plain}
\newtheorem{cor}[stmcounter]{Corollary}

\newtheorem{thm}[stmcounter]{Theorem}

\newtheorem{prop}[stmcounter]{Proposition}
\newtheorem{lem}[stmcounter]{Lemma}
\newtheorem{probl}[stmcounter]{Problem}

\theoremstyle{definition}
\newtheorem{defin}[stmcounter]{Definition}

\theoremstyle{remark}
\newtheorem{ex}[stmcounter]{Example}
\newtheorem{rem}[stmcounter]{Remark}
\newtheorem{con}[stmcounter]{Construction}

\begin{document}

\title{Manifolds of isospectral arrow matrices}

\author{Anton Ayzenberg, Victor Buchstaber}
\address{Faculty of computer science, Higher School of Economics}
\email{ayzenberga@gmail.com}
\address{V.A. Steklov Mathematical Institute, RAS, Moscow, Russia}
\email{buchstab@mi.ras.ru}

\date{\today}
\thanks{The publication was prepared within the framework of the Academic
Fund Program at the National Research University Higher School of
Economics (HSE) in 2018-2019 (grant  N 18-01-0030) and by the
Russian Academic Excellence Project ``5-100''.}
\subjclass[2010]{Primary 52B11, 15A42, 57R19, 57R91; Secondary
05E45, 52B70, 15B57, 52C45, 55N91, 57S25, 20Bxx, 53D20}
\keywords{matrix spectrum, Hermitian matrix, sparse matrix, torus
action, permutation action, moment map, fundamental domain,
manifold with corners, codimension two face cuts}

\begin{abstract}
An arrow matrix is a matrix with zeroes outside the main diagonal,
first row, and first column. We consider the space
$M_{\St_n,\lambda}$ of Hermitian arrow $(n+1)\times
(n+1)$-matrices with fixed simple spectrum $\lambda$. We prove
this space to be a smooth $2n$-manifold, and its smooth structure
is independent on the spectrum. Next, this manifold carries the
locally standard torus action: we describe the topology and
combinatorics of its orbit space. If $n\geqslant 3$, the orbit
space $M_{\St_n,\lambda}/T^n$ is not a polytope, hence
$M_{\St_n,\lambda}$ is not a quasitoric manifold. However, there
is a natural permutation action on $M_{\St_n,\lambda}$ which
induces the combined action of a semidirect product
$T^n\rtimes\Sigma_n$. The orbit space of this large action is a
simple polytope $\B^n$. The structure of this polytope is
described in the paper.

In case $n=3$, the space $M_{\St_3,\lambda}/T^3$ is a solid torus
with boundary subdivided into hexagons in a regular way. This
description allows to compute the cohomology ring and equivariant
cohomology ring of the 6-dimensional manifold $M_{\St_3,\lambda}$
using the general theory developed by the first author. This
theory is also applied to a certain $6$-dimensional manifold
called the twin of $M_{\St_3,\lambda}$. The twin carries a
half-dimensional torus action and has nontrivial tangent and
normal bundles.
\end{abstract}

\maketitle


\section{Introduction}\label{secIntro}

Spaces of isospectral Hermitian or symmetric matrices lie in the
focus of several areas of mathematics, including symplectic
geometry, representation theory, toric topology, and applied
mathematics.

Let $M_{n+1}$ be the space of all Hermitian matrices of size $n+1$
and let $M_\lambda\subset M_{n+1}$ denote the subspace of all
Hermitian matrices with the fixed simple spectrum
$\lambda=\{\lambda_0,\lambda_1,\ldots,\lambda_n\}$ (we assume
$\lambda_0<\lambda_1<\cdots<\lambda_{n}$). The unitary group
$U(n+1)$ acts on $M_{n+1}$ by conjugation. Multiplying Hermitian
matrices by $\sqrt{-1}$, we get skew Hermitian matrices, hence the
action can be identified with the adjoint action of $U(n+1)$ on
its tangent Lie algebra. For a simple spectrum $\lambda$ the
subset $M_\lambda$ is the principal orbit of this action. The
manifold $M_\lambda$ is diffeomorphic to the variety of full
complex flags $\Fl_{n+1}=U(n+1)/T^{n+1}$. Here
\[
T^{n+1}=\left\{D=\diag(t_0,\ldots,t_{n})\mid t_i\in \Co,
|t_i|=1\right\}
\]
is the maximal compact torus consisting of diagonal unitary
matrices. The torus acts on both matrix spaces $M_{n+1}$ and
$M_\lambda$ by matrix conjugation $A\mapsto DAD^{-1}$. In
coordinate form we have
\begin{equation}\label{eqCoordMain}
(a_{ij})_{\substack{i=0,\ldots,n\\j=0,\ldots,n}}\mapsto
(t_it_j^{-1}a_{ij})_{\substack{i=0,\ldots,n\\j=0,\ldots,n}}.
\end{equation}
It is reasonable to look for subspaces in the flag manifold
$M_\lambda\cong \Fl_{n+1}$, which are preserved by the torus
action. The formula \eqref{eqCoordMain} implies that the torus
action preserves zeroes at given positions. Hence, one can study
the spaces of Hermitian matrices, with the given spectrum and
zeroes at prescribed positions.

The classical example is the space of $M_{\I_{n},\lambda}$ of
isospectral tridiagonal Hermitian matrices. This space was studied
in \cite{Tomei,BFR,DJ}. Tomei \cite{Tomei} considered the real
analogue of this space: he proved that this space is a smooth
manifold, and its smooth type is independent of the spectrum.
Bloch--Flaschka--Ratiu \cite{BFR} studied the Hermitian case and
demonstrated its connection with the toric variety of type $A_n$.
General theory developed in the seminal work of Davis and
Januszkiewicz \cite{DJ} allows to describe cohomology ring and
$T^n$-equivariant cohomology ring of $M_{\I_{n},\lambda}$.

Instead of tridiagonal matrices one can consider staircase
Hermitian matrices (also known as generalized Hessenberg
matrices), see \cite{Nan,dMP}. The nonzero elements of such
matrices are allowed only in the vicinity of the diagonal which is
encoded by the so called Hessenberg function. The spaces of
Hermitian staircase matrices can be studied similarly to
tridiagonal case: the properties of generalized Toda flow can be
used to prove the smoothness of these spaces. We collected the
results on such ``matrix Hessenberg manifolds'' in~\cite{ABhess}.

Another way to generalize tridiagonal matrices is to allow two
additional non-zero entries at top right and bottom left corner of
the matrix. Such matrices are called periodic tridiagonal
matrices. They appear in the study of discrete Schr\"{o}dinger
operator in mathematical physics (see \cite{VanM,Krich}). The
study of the isospectral space of such matrices is done in the
forthcoming paper \cite{AyzMatr}.

In this paper we study the isospectral space $M_{\St_n,\lambda}$
of matrices which have zeroes outside the diagonal, first row, and
first column. Matrices of this form will be called \emph{arrow
matrices}. We are indebted to Tadeusz Januszkiewicz \cite{Jan} for
telling us about this wonderful object. The forthcoming paper by
Januszkiewicz and Gal will reveal homological and symplectic
features of this isospectral space in general.

In this paper we prove that $M_{\St_n,\lambda}$ is smooth and its
smooth type is independent of $\lambda$, see Section
\ref{secArrowMatricesOrbits}. The action of a torus $T=T^n$ on
$M_{\St_n,\lambda}$ is locally standard, so the orbit space
$Q_n=M_{\St_n,\lambda}/T$ is a manifold with corners. In Section
\ref{secArrowMatricesOrbits} we describe the topology of the orbit
space. We introduce a cubical complex $\Sq_n$ which is the union
of cubical faces of $n$-dimensional permutohedron, and prove that
the orbit space $Q_n$ is homotopy equivalent to $\Sq_{n-1}$. It
follows that for $n\geqslant 3$, the orbit space is not a simple
polytope. The combinatorial face structure of $Q_n$ is described
in Section \ref{secTreeMatricesCombinatorics} with a general
notion of a cluster-permutohedron. The family of
cluster-permutohedra contains two known examples: a permutohedron
and a cyclopermutohedron of Gaiane Panina, and provide interesting
examples of partially ordered sets, which, as far as we know, have
not been considered in combinatorial geometry yet.

In general, there is a natural permutation action of the symmetric
group $\Sigma_n$ on the manifold $M_{\St_n,\lambda}$ as well as on
the orbit space $Q_n$. We show that the fundamental domain of the
$\Sigma_n$-action on $Q_n$ is a simple polytope; in Section
\ref{secArrowMatricesBlock} we describe this polytope. This
description allows to reconstruct $Q_n$ by stacking together $n!$
copies of this polytope. In the future we hope this description of
the orbit space $Q_n$ will allow to construct effective
diagonalization algorithms for arrow-shaped matrices.

We are especially interested in arrow matrices of size $4\times
4$, that is the case $n=3$. In this case the orbit space
$Q_3=M_{\St_3,\lambda}$ is a solid torus whose boundary is
subdivided into hexagons in a regular way. We knew this fact from
Tadeusz Januszkiewicz \cite{Jan}. The cohomology and equivariant
cohomology rings of the space $M_{\St_n,\lambda}$ itself can be
computed with the theory developed by the first author in
\cite{Ay0,Ay1,Ay2,Ay3,AMPZ}. This theory, in general, allows to
describe cohomology and equivariant cohomology rings of manifolds
with locally standard torus action, whose orbit spaces have only
acyclic proper faces. Since every facet of $Q_3$ is a hexagon, we
are in position to apply this theory to $M_{\St_3,\lambda}$. In
Section \ref{secArrowMatricesTopology} we recall the notions of
Stanley--Reisner ring, $h$-, $h'$-, and $h''$-numbers of
simplicial complexes, Novik--Swartz theory, and related
topological results of \cite{DJ,MasPan}. Theorem
\ref{thmM3cohomology}, based on these results, describes the
homological structure of the manifold $M_{\St_3,\lambda}$.

In the last Section \ref{secTwin} we introduce the manifold $X_n$
whose properties are similar to that of $M_{\St_3,\lambda}$. This
manifold also carries a half-dimensional torus action and its
orbit space is isomorphic to $Q_n$. However, the manifold $X_n$ is
more interesting from the topological point of view. In case $n=3$
we describe the cohomology ring and show that the first Pontryagin
class of $X_3$ is nonzero.

\section{Spaces of sparse isospectral matrices}\label{secGeneral}

The action of $T^{n+1}$ on the isospectral space $M_\lambda$ is
noneffective, since the scalar matrices act trivially. Hence we
have an effective action of $T=T^n\cong T^{n+1}/\Delta(T^1)$.
Fixed points of the torus action on $M_{\lambda}$ are the diagonal
matrices with the spectrum $\lambda$, i.e. the matrices of the
form
$\diag(\lambda_{\sigma(0)},\lambda_{\sigma(1)},\ldots,\lambda_{\sigma(n)})$
for all possible permutations $\sigma\in \Sigma_{n+1}$.

The action is hamiltonian. Indeed, $M_\lambda$ can be identified
with the orbit of (co)adjoint action of $U(n+1)$ on its
(co)tangent algebra. This orbit possess Kostant--Kirillov
symplectic form, and the action of $U(n+1)$ (hence $T^{n+1}$) on
this orbit is hamiltonian. The momentum map for the torus action
is given by
\[
\mu\colon M_{\lambda} \to \Ro^{n+1},\quad A=(a_{i,j})\mapsto
(a_{0,0},a_{1,1},\ldots,a_{n,n})
\]
(the image of this map lies in the hyperplane $\{\sum_{i=0}^n
a_{i,i}=\sum\lambda_i=\const\}\cong\Ro^n$).
Atiyah--Guillemin--Sternberg theorem tells that the image of the
moment map is the permutohedron
$\conv\{(\lambda_{\sigma(0)},\lambda_{\sigma(1)},\ldots,\lambda_{\sigma(n)})\mid
\sigma\in\Sigma_{n+1}\}$. We will call it Schur--Horn's
permutohedron, since the description of diagonals of all Hermitian
matrices with the given spectrum is the classical result of Schur
and Horn.

\begin{con}
Let $\Gamma=(V,E)$ be a graph without multiple edges and loops on
the vertex set $V=\{0,1,\ldots,n\}$. Consider the vector subspace
of Hermitian matrices:
\[
M_\Gamma=\{A\in M_{n+1}\mid a_{ij}=0, \mbox{ если } \{i,j\}\notin
E\}.
\]
As noted in the introduction, the torus action preserves the set
$M_\Gamma$ for any $\Gamma$. We set
\[
M_{\Gamma,\lambda}=M_{\Gamma}\cap M_{\lambda}.
\]
The action of $T^n$ can be restricted to $M_{\Gamma,\lambda}$. The
space $M_{\Gamma,\lambda}$ is called the \emph{space of
isospectral sparse matrices of type} $\Gamma$. We have
\begin{equation}\label{eqDimM} \dim
M_{\Gamma,\lambda}=2|E|
\end{equation}
\end{con}

\begin{ex}
If $\Gamma$ is a complete graph $\{0,\ldots,n\}$, then
$M_{\Gamma,\lambda}=M_\lambda\cong \Fl_{n+1}$.
\end{ex}

\begin{rem}\label{remBlockMatrices}
Without loss of generality, only connected graphs can be
considered. Let $\Gamma_1,\ldots,\Gamma_k$ be the connected
components of a graph $\Gamma$ with the vertex sets
$V_1,\ldots,V_k\subset \{0,\ldots,n\}$ respectively. Let $\Omega$
be the set of all possible partitions of the set
$\{\lambda_0,\ldots,\lambda_n\}$ into disjoint subsets $S_i$ of
cardinalities $|A_i|$, $i=1,\ldots,k$. Then
$M_{\Gamma,\lambda}=\bigsqcup_\Omega \prod_{i=1}^k
M_{\Gamma_i,S_i}$. We discuss the underlying combinatorial
structures in detail in Section
\ref{secTreeMatricesCombinatorics}.
\end{rem}

\begin{probl}
Is it true that for each $\Gamma$ the subspace
$M_{\Gamma,\lambda}$ is a smooth manifold, which smooth type is
independent of $\lambda$?
\end{probl}

\begin{rem}
Kirillov's form can be restricted to $M_{\Gamma,\lambda}$, however
the restriction need not be symplectic even if
$M_{\Gamma,\lambda}$ is smooth. Therefore,
Atiyah--Guillemin--Sternberg theorem is no longer applicable in
the general case. Still there is a map $\mu\colon
M_{\Gamma,\lambda}\to \Ro^{n+1}$ taking the diagonal of the
matrix. This map is constant on each torus orbit, hence there is
an induced map $\tilde{\mu}\colon M_{\Gamma,\lambda}/T^n\to
\Ro^{n+1}$.
\end{rem}

\section{Tree matrices}\label{secTreeMatricesPreliminaries}

Let $\Gamma$ be a tree on the vertex set $\{0,1,\ldots,n\}$. In
this case the elements of $M_\Gamma$ will be called \emph{tree
matrices}. The $2n$-dimensional space $M_{\Gamma,\lambda}$ carries
an effective action of a compact $n$-torus. This fact makes tree
matrices an important object: they produce natural examples of
half-dimensional torus actions. The moment map
\[
\mu\colon M_{\Gamma,\lambda}\to H\subset \Ro^{n+1},\quad H=\{\sum
a_{i,i}=\const\}
\]
induces the map of $n$-dimensional orbit space
$M_{\Gamma,\lambda}/T$ into Schur--Horn permutohedron
$\Pe_\lambda^n\subset H\cong \Ro^n$. All vertices of
$\Pe_\lambda^n$ belong to $\mu(M_{\Gamma,\lambda})$. In the
following it will be convenient to encode tree matrices in terms
of labeled trees.

\begin{defin}
A \emph{labeled tree} $\Delta$ is a triple $\Delta=(\Gamma,a,b)$,
where $\Gamma=(V,E)$ is a tree, $a\colon V\to \Ro$, $b\colon E\to
\Co$.
\end{defin}

In other words, a labeled tree is a tree with a real number $a_i$
assigned to each vertex $i\in V$, and a complex number $b_e$
assigned to each edge $e\in E$. A labeled tree determines the
Hermitian matrix $A_\Delta$ as follows. The diagonal elements of
$A_\Delta$ are $a_i$ at diagonal, and $(i,j)$-th element is zero
if $\{i,j\}$ is not an edge of $\Gamma$; $b_e$ if $i<j$ and
$e=\{i,j\}$; and $\overline{b_e}$ if $j<i$ and $e=\{j,i\}$.

If, moreover, some vertex $k$ is fixed in a labeled tree, we call
it a \emph{rooted labeled tree} with the \emph{root} $k$.

\begin{ex}\label{exTridiagonal}
Let $\I_n$ denote the path graph with the edges
$E=\{\{0,1\},\{1,2\},\ldots,\{n-1,n\}\}$. Then $M_{\I_n,\lambda}$
is the space of tridiagonal isospectral Hermitian matrices. The
classical result (see \cite{Tomei,BFR}) states that
$M_{\I_n,\lambda}$ is a smooth manifold, and its smooth type is
independent of $\lambda$. It follows from the result of Tomei that
the orbit space $M_{\I_n, \lambda}/T^n$ is diffeomorphic to
$n$-dimensional permutohedron as a manifold with corners.

The moment map $\mu\colon \Pe^n\cong M_{\Gamma, \lambda}/T^n\to
\Pe_\lambda^n$ is not the isomorphism of permotohedra. This map
determines the bijection between vertices of permutohedra, however
this map is neither injective nor surjective on the interior of a
permutohedron. In case $n=2$ the image of the moment map is shown
on Fig.\ref{pictTwistedHex} (this fact is proved in Proposition
\ref{propMomentImage} below).
\end{ex}

\begin{figure}[h]
\begin{center}
\includegraphics[scale=0.3]{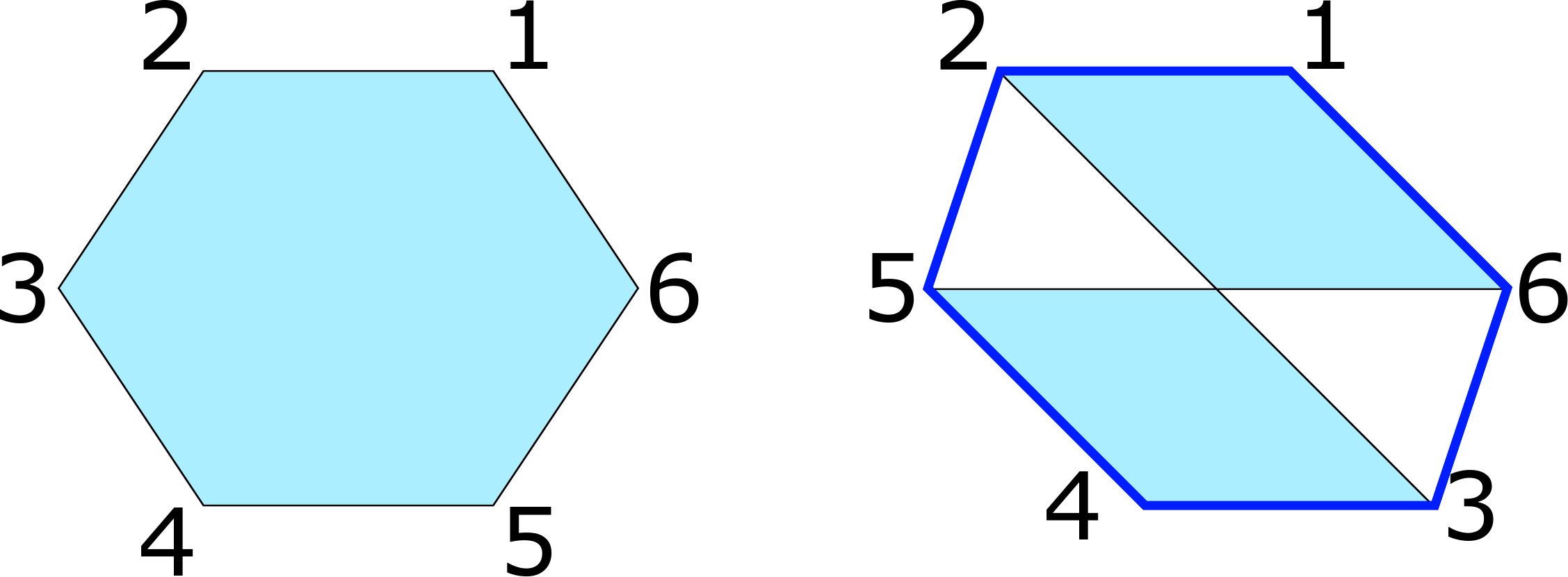}
\end{center}
\caption{Image of the moment map for tridiagonal
$3\times3$-matrices}\label{pictTwistedHex}
\end{figure}


\begin{ex}
Let $\St_n$ denote the star graph with the edge set
$E=\{\{0,1\},\{0,2\},\ldots,\{0,n\}\}$. In this case, matrices
from $M_{\St_n}$ have the form
\begin{equation}\label{eqArrowMatrixView}
A_\Delta=\begin{pmatrix}
a_0 & b_1 & \cdots & b_n\\
\overline{b_1}& a_1 & 0&0\\
\vdots& \vdots & \ddots & \vdots\\
\overline{b_n}&0&\cdots&a_n
\end{pmatrix}
\end{equation}
Such matrices are called \emph{arrow matrices}.
\end{ex}

We formulate several technical statements about general tree
matrices. At first, there is a natural notion of a tree fraction,
which generalizes that continued fractions.

\begin{defin}
Let $\Delta=(\Gamma=(V,E),a,b)$ be a rooted labeled tree with root
$k\in V$. Define the tree fraction $Q(\Delta,k)$ associated with
$(\Delta,k)$ by recursion:
\begin{enumerate}
\item Let $\Delta$ be a labeled tree with root $k$
and at least one more vertex. By deleting $k$ from $\Delta$ the
tree breaks down into $s$ connected components, where each
connected component is a rooted labeled tree $\Delta_i$, whose
root $k_i$ is a descendant of $k$. Let $b_1,\ldots,b_s\in \Co$ be
the labels on the edges of $\Delta$ connecting $k$ with its
descendants. Set
\[
Q(\Delta,k)=a_k-\sum_{i=1}^s\dfrac{|b_i|^2}{Q(\Delta_i,k_i)}.
\]

\item If $\Gamma$ has a single vertex $k$ labeled by
$a_k\in \Ro$, then we set $Q(\Delta,k)=a_k$.
\end{enumerate}
\end{defin}

In particular, for a labeled path graph, rooted at the endpoint,
\[
\begin{tikzpicture}
\draw (0,0)--(4,0);

\filldraw [red] (0,0) circle (2pt);

\draw (0,0.3) node{$a_0$}; \draw (0.5,-0.3) node{$b_1$};

\filldraw [black] (1,0) circle (1pt);

\draw (1,0.3) node{$a_1$}; \draw (1.5,-0.3) node{$b_2$};

\filldraw [black] (2,0) circle (1pt);

\draw (2,0.3) node{$a_2$};

\draw (3,0.3) node{$\cdots$};

\filldraw [black] (4,0) circle (1pt);

\draw (4,0.3) node{$a_n$};
\end{tikzpicture}
\]
the tree fraction $Q(\Delta,k)$ is the continued fraction of the
form
\[
a_0-\cfrac{|b_1|^2}{a_1-\cfrac{|b_2|^2}{\ddots\cfrac{\vdots}{a_{n-1}-\cfrac{|b_n|^2}{a_n}}}}
\]

\begin{lem}\label{lemInverseElement}
Let $\Delta$ be a labeled tree and $A_\Delta$ the corresponding
Hermitian matrix. The diagonal elements of the inverse matrix are
given by the tree fractions:
\[
(A_\Delta^{-1})_{k,k}=\dfrac{1}{Q(\Delta,k)}.
\]
\end{lem}

\begin{proof}
The proof is an exercise in linear algebra.
\end{proof}

We say that $S$ is the splitting of the tree $\Gamma$, if $S$ is
the partition of the set of vertices into 1- and 2-element
subsets, in which all 2-element subsets are the edges of $\Gamma$.
Let $S(\Gamma)$ be the set of all splittings of $\Gamma$. For
$S\in S(\Gamma)$ set $\sigma(S)=(-1)^p$ where $p$ is the number of
edges in the splitting.

\begin{lem}\label{lemDet}
For a labeled tree $\Delta=(\Gamma,a,b)$ there holds
\[
\det A_\Delta=\sum_{S\in S(\Gamma)}\sigma(S)\prod_{i\mbox{ vertex
of } S}a_i\prod_{e\mbox{ edge of } S}|b_e|^2.
\]
\end{lem}

\begin{proof}
Another exercise in linear algebra. Expand the determinant along
the row corresponding to a leaf vertex of the tree.
\end{proof}

\begin{cor}\label{corDetOfArrow}
For an arrow matrix (which corresponds to the star graph) we have
\[
\det\begin{pmatrix}
a_0 & b_1 & \cdots & b_n\\
\overline{b_1}& a_1 & \ldots &0\\
\vdots& \vdots & \ddots & \vdots\\
\overline{b_n} & 0 & \cdots & a_n
\end{pmatrix}=a_0a_1\cdots a_n-\sum_{i=1}^n|b_i|^2a_1\cdots
\widehat{a_i}\cdots a_n.
\]
The star graph is the only graph for which $\det(A_\Delta)$ is
quadratic in variables $b_i$.
\end{cor}

\section{Space of isospectral arrow-matrices}\label{secArrowMatricesOrbits}

Let $\Delta=(\St_n,a,b)$ be a labeled star graph and $A_\Delta$ be
the corresponding arrow matrix given by \eqref{eqArrowMatrixView}.
Consider the isospectral space
$M_{\St_n,\lambda}=\{A_\Delta\mid \Spec A_\Delta=\lambda\}$. 
We describe the image of the moment map of such matrices, that is
the set of all possible diagonals $(a_0,\ldots,a_n)\in\Ro^{n+1}$.
Since $a_0=\sum_{i=0}^n\lambda_i-\sum_{i=1}^na_i$, it is
sufficient to describe all possible $(a_1,\ldots,a_n)\in \Ro^n$.

\begin{prop}\label{propMomentImage}
Let $I_j=[\lambda_{j-1},\lambda_j]\subset\Ro$, for $j=1,\ldots,n$.
Then $\mu(M_{\St_n,\lambda})$ is the set
\[
\{(a_0,a_1,\ldots,a_n)\mid (a_1,\ldots,a_n)\in \ca{R}_n,\quad
a_0=\sum\nolimits_{i=0}^n\lambda_i-\sum\nolimits_{i=1}^na_i\},
\]
where
\[
\ca{R}_n=\bigcup_{\sigma\in
\Sigma_n}I_{\sigma(1)}\times\cdots\times I_{\sigma(n)}\subset
\Ro^n
\]
is the union of $n!$ cubes of dimension $n$.
\end{prop}

\begin{proof}
First note that the permutation group $\Sigma_n$ acts on the star
graph $\Gamma$ by permuting its rays. As a corollary, there is an
action of $\Sigma_n$ on the vector space $M_{\St_n}$. The
permutation action preserves the spectrum, hence there is an
induced $\Sigma_n$-action on $M_{\St_n,\lambda}$. Therefore, we
may assume $a_1\leqslant a_2\leqslant\cdots\leqslant a_n$.

Let us prove that under condition
\[
\lambda_0<a_1<\lambda_1<a_2<\lambda_2<\cdots<a_n<\lambda_n,\qquad
a_0=\sum\nolimits_{i=0}^n\lambda_i-\sum\nolimits_{i=1}^na_i
\]
there exists an arrow matrix $A_\Delta$ with the diagonal
$(a_0,\ldots,a_n)$ and eigenvalues $\lambda_0,\ldots,\lambda_n$.
This would imply that the interior of $\ca{R}_n$ lies in the image
of the moment map. Consider the polynomials
$P(\lambda)=\prod_{i=0}^n(\lambda-\lambda_i)$,
$Q(\lambda)=\prod_{i=1}^n(\lambda-a_i)$. Division by $Q(\lambda)$
yields
\begin{equation}\label{eqDivis}
P(\lambda)=(\lambda-\alpha)Q(\lambda)+R(\lambda).
\end{equation}
It is easy to check that
$\alpha=\sum\nolimits_{i=0}^n\lambda_i-\sum\nolimits_{i=1}^na_i=a_0$.
Substituting all possible $a_i$ into \eqref{eqDivis}, we get
\[
R(a_n)=P(a_n)<0, \quad R(a_{n-1})=P(a_{n-1})>0, \quad
R(a_{n-2})=P(a_{n-2})<0,\quad\ldots
\]
For the rational function $\frac{P(\lambda)}{Q(\lambda)}$ we have
the partial fraction expansion
\[
\dfrac{P(\lambda)}{Q(\lambda)}=\lambda-a_0 +
\dfrac{r_1}{\lambda-a_1}+\cdots+\dfrac{r_n}{\lambda-a_n},
\]
where
\[
r_i=\dfrac{R(a_i)}{\prod_{j\neq i}(a_i-a_j)}<0\quad\mbox{for all }
i=1,\ldots,n.
\]
Hence we can put $r_i=-|b_i|^2$ for some $b_i\in\Co$. Consider the
arrow matrix $A_\Delta$ of the form~\eqref{eqArrowMatrixView}.
According to Lemma \ref{lemInverseElement}, the top left element
of the matrix $(A_\Delta-\lambda E)^{-1}$ can be written as the
tree fraction
\[
(A_\Delta^{-1})_{0,0}=\cfrac{1}{a_0-\lambda-\cfrac{|b_1|^2}{a_1-\lambda}-\cdots-
\cfrac{|b_n|^2}{a_n-\lambda}}=-\dfrac{1}{P/Q}=-\dfrac{\prod_{i=1}^n(\lambda-a_i)}{\prod_{i=0}^n(\lambda-\lambda_i)}.
\]
This meromorphic function has poles at the points
$\lambda_0,\ldots,\lambda_n$. On the other hand the function
$(A_\Delta-\lambda E)^{-1}$ is holomorphic outside the spectrum of
$A_\Delta$. Therefore $\lambda_i$ are the eigenvalues of
$A_\Delta$. We remark that the similar technique was applied by
Moser \cite{Moser} in the study of tridiagonal matrices: he
attributes this technique to Stieltjes.

The converse reasoning shows that for every arrow matrix
$A_\Delta$ its eigenvalues and elements $a_1,\ldots,a_n$
alternate. We give another, even more elementary proof of this
fact. Let $a_1<a_2<\cdots<a_n$ be fixed. The eigenvalues are the
roots of the equation $\det(A_\Delta-\lambda E)~=~0$. Applying
Corollary \ref{corDetOfArrow}, we get the equations
\begin{equation}\label{eqEigen1}
|b_1|^2\prod_{i\neq 0,1}(a_i-\lambda)+\ldots+|b_n|^2\prod_{i\neq
0,n}(a_i-\lambda)=\prod_{i=0}^n(a_i-\lambda);
\end{equation}
\begin{equation}\label{eqEigen2}
\dfrac{|b_1|^2}{a_1-\lambda}+\ldots+\dfrac{|b_n|^2}{a_n-\lambda}=a_0-\lambda.
\end{equation}
The function on the left side of \eqref{eqEigen2} increases on
every interval between the poles, and the function on the right
side decreases. Therefore the equation has a unique root at each
of interval
\[
(-\infty,a_1), (a_1,a_2),\ldots,(a_{n-1},a_n), (a_n,\infty),
\]
Hence the numbers $a_i$ and the eigenvalues alternate. When some
of the values $a_i$ coincide, e.g. $a_j=a_{j+1}=\cdots=a_{j+r}=a$,
then, when passing from \eqref{eqEigen1} to \eqref{eqEigen2}, we
loose the root $\lambda=a$ of multiplicity $r$. Since the spectrum
is assumed simple, the numbers $a_i$, $i=1,\ldots,n$ are allowed
to have at most double collisions. Even in this case the values
$a_i$ and eigenvalues alternate non-strictly.

This reasoning shows that points outside $\ca{R}_n$ do not lie in
the image of the moment map. This completes the proof.
\end{proof}

\begin{ex}
In case $n=2$ the image of $\mu$ consists of two squares, sitting
inside a hexagon. Arrow matrices of size $3\times 3$ coincide with
tridiagonal matrices of this size up to permutation of rows and
columns. This explains Example \ref{exTridiagonal} and Fig.
\ref{pictTwistedHex}.
\end{ex}

\begin{ex}
In case $n=3$ the moment map image $\mu(M_{\St_n,\lambda})$ is the
union of $6$ cubes shown on Fig.\ref{pictNecklace}. The contour
shows the enveloping convex hull of these cubes, which is the
Schur--Horn permutohedron.
\end{ex}

\begin{figure}[h]
\begin{center}
\includegraphics[scale=0.3]{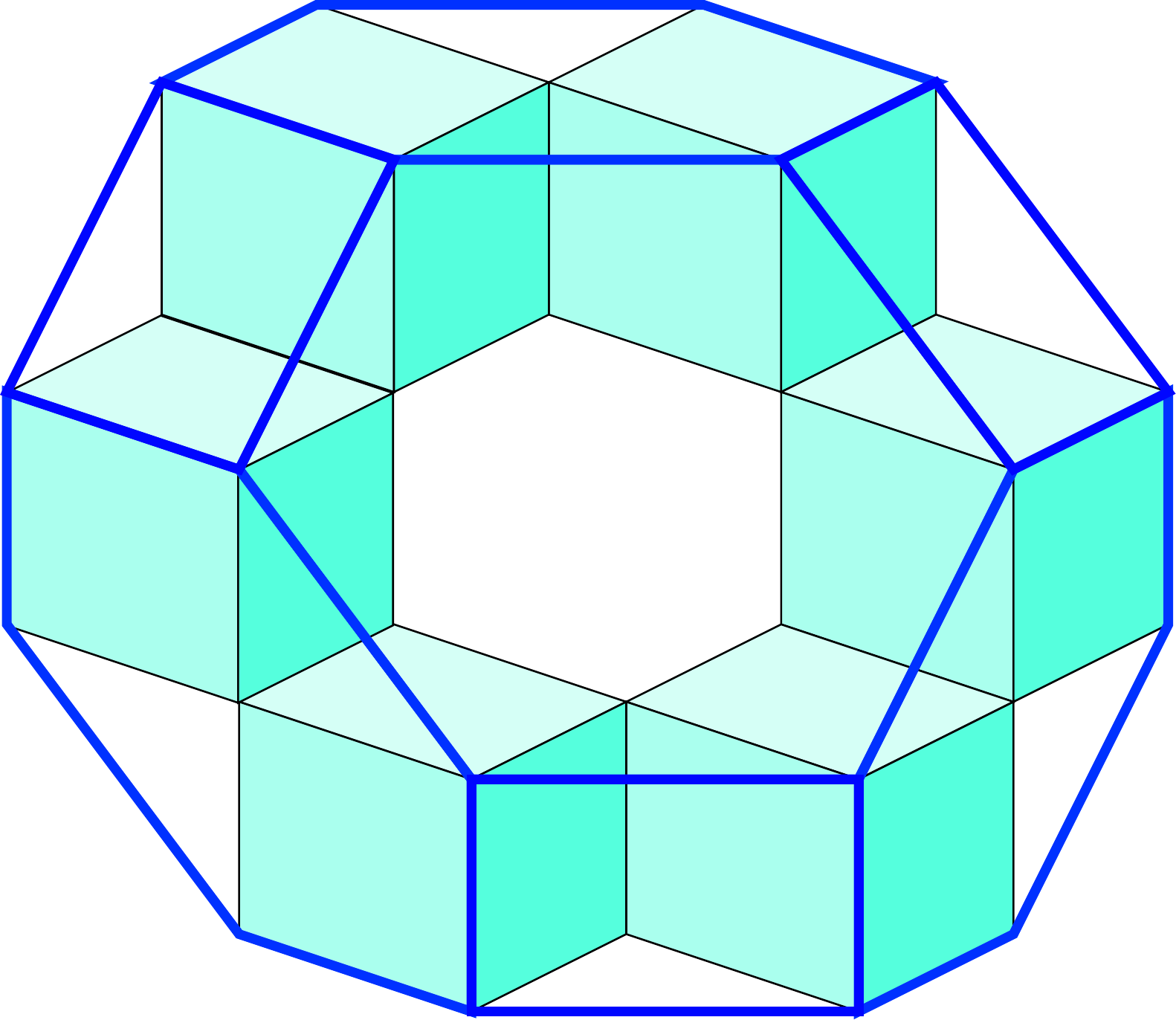}
\end{center}
\caption{The moment map image for the star graph
$\St_3$}\label{pictNecklace}
\end{figure}

\begin{thm}\label{thmIsSmooth}
The space $M_{\St_n,\lambda}$ is a smooth manifold of dimension
$2n$. The space $M_{\St_n,\lambda}^{\Ro}$ of isospectral real
symmetric arrow matrices is a smooth manifold of dimension $n$.
The action of $T^n$ on $M_{\St_n,\lambda}$ and the action of
$\Zt^n$ on $M_{\St_n,\lambda}^\Ro$ are locally standard.
\end{thm}

\begin{proof}
We need a technical statement, well known in linear algebra.

\begin{lem}\label{lemMatrixFractions}
Let $d_i,\nu_i$, $i\in\{1,\ldots,n\}$ be two sets of numbers, such
that all $2n$ numbers are distinct. Then the square matrix
$B=(B_{i,j}=\frac{1}{d_i-\nu_j})_{1\leqslant i,j\leqslant n}$ is
invertible.
\end{lem}

\begin{proof}
The statement follows from the uniqueness of the partial fraction
expansion. Assume that $B(c_1,\ldots,c_n)^\top=0$ for a vector
$(c_1,\ldots,c_n)\in \Ro^n$. Then the rational function
\[
R(d)=\dfrac{c_1}{d-\nu_1}+\cdots+\dfrac{c_n}{d-\nu_n}
\]
has roots in $n$ points $d_1,\ldots,d_n$. Since the degree of the
numerator of $R(d)$ is less than $n$, $R(d)$ is identically zero
and we have $c_i=0$ for all $i=1,\ldots,n$. Thus $B$ is
invertible.
\end{proof}

Let us prove the theorem. The subspace $M_{\St_n,\lambda}$ is
given in the vector space $M_{\St_n}$ by the equations
\[
P_j(\underline{a},\underline{b})=|b_1|^2\prod_{i\neq
0,1}(a_i-\lambda_j)+\ldots+|b_n|^2\prod_{i\neq
0,n}(a_i-\lambda_j)-\prod_{i=0}^n(a_i-\lambda_j)=0;\qquad
j\in[n]=\{1,\ldots,n\}
\]
\[
\sum_{i=0}^n a_i-\sum_{i=0}^n\lambda_i=0.
\]
according to the Corollary \ref{corDetOfArrow}. We use the vector
notation
\[
\dfrac{\dd P_j}{\dd b}=\left(\dfrac{\dd P_j}{\dd
b_1},\ldots,\dfrac{\dd P_j}{\dd b_n}\right);\qquad \dfrac{\dd
P_j}{\dd a}=\left(\dfrac{\dd P_j}{\dd a_1},\ldots,\dfrac{\dd
P_j}{\dd a_n}\right).
\]
It is sufficient to show that the vectors $(\frac{\dd P_j}{\dd
b},\frac{\dd P_j}{\dd a})$, $j\in[n]$ are linearly independent at
all points of $M_{\St_n,\lambda}$. We have
\[
\dfrac{\dd P_j}{\dd b_k}=\overline{b_k}\prod_{i\neq
0,k}(a_i-\lambda_j);
\]
\begin{multline}\label{eqPartialB}
\dfrac{\dd P_j}{\dd b}=\left(\overline{b_1}\prod_{i\neq
0,1}(a_i-\lambda_j),\overline{b_2}\prod_{i\neq
0,2}(a_i-\lambda_j),\ldots,\overline{b_n}\prod_{i\neq
0,n}(a_i-\lambda_j)\right)=\\=\prod_{i=1}^n(a_i-\lambda_j)\left(\dfrac{\overline{b_1}}{a_1-\lambda_j},
\ldots, \dfrac{\overline{b_n}}{a_n-\lambda_j}\right).
\end{multline}

First consider the general case: let all $a_i$, $i\in[n]$ be
distinct. As proved in Proposition~\ref{propMomentImage}, we may
assume
\[
\lambda_0<a_1<\lambda_1<a_2<\cdots<a_n<\lambda_n,
\]
and $b_i\neq 0$ for all $i\in[n]$.

It follows from \eqref{eqPartialB} that the matrix formed by the
vectors $\frac{\dd P_j}{\dd b}$, $j=1,\ldots,n$, has the form
\[
\left(\dfrac{\dd P_j}{\dd
b}\right)=\prod_{i=1}^n\overline{b_i}\prod_{i\neq
j}(a_i-\lambda_j)\begin{pmatrix}
\frac{1}{a_1-\lambda_1}&\cdots&\frac{1}{a_n-\lambda_1}\\
\vdots&\ddots&\vdots\\
\frac{1}{a_1-\lambda_n}&\cdots&\frac{1}{a_n-\lambda_n}
\end{pmatrix}
\]
This matrix is nondegenerate according to Lemma
\ref{lemMatrixFractions}. This proves the smoothness of
$M_{\St_n,\lambda}$ at generic points.

Now we allow some of the points $\{a_i\}$ collide. As noted in the
proof of Proposition~\ref{propMomentImage}, only pairwise
collisions may occur:
\[
\cdots<a_{j_1}=a_{j_1+1}<\cdots<a_{j_2}=a_{j_2+1}<\cdots<a_{j_s}=a_{j_s+1}<\cdots
\]
and each pair of collided diagonal elements determines the
eigenvalue $\lambda_{j_l}=a_{j_l}=a_{j_l+1}$. All other
eigenvalues still lie on the open intervals between $a_i$. We
denote by $F$ the set of all eigenvalues, lying between $a_i$, and
by $D$ the set of eigenvalues which come from collided diagonal
elements. We have $D=\{j_1,\ldots,j_s\}$ and
$F=\{0,\ldots,n\}\setminus D$.

Let $A\in M_\Gamma$ be a matrix such that
$a_{j_l}=a_{j_l+1}=\lambda_{j_l}$. Simple computation shows that
$\dfrac{\dd P_{j_l}}{\dd b}(A)=0$. Moreover at a point $A$ we have
\[
\dfrac{\dd P_{j_l}}{\dd a_j}(A)=0
\]
if $j\neq j_l,j_l+1$, and
\[
\dfrac{\dd P_{j_l}}{\dd a_{j_l}}(A)=|b_{j_l+1}|^2\prod_{i\neq
j_l,j_l+1}(a_i-\lambda_{j_l});\qquad \dfrac{\dd P_{j_l}}{\dd
a_{j_l+1}}(A)=|b_{j_l}|^2\prod_{i\neq
j_l,j_l+1}(a_i-\lambda_{j_l})
\]
Similar reasoning as before shows that
$-(|b_{j_l}|^2+|b_{j_l+1}|^2)$ is the coefficient of
$\frac{1}{\lambda-a_{j_l}}$ in the partial fraction expansion of
the (reducible) fraction
$\frac{\prod_{i=0}^n(\lambda-\lambda_i)}{\prod_{i=1}^n(\lambda-a_i)}$.
Hence $|b_{j_l}|^2+|b_{j_l+1}|^2\neq 0$ and one of the numbers
$\frac{\dd P_{j_l}}{\dd a_{j_l}}(A)$, $\frac{\dd P_{j_l}}{\dd
a_{j_l+1}}(A)$ is nonzero. Without loss of generality assume that
$\frac{\dd P_{j_l}}{\dd a_{j_l}}(A)\neq 0$ for all $l=1,\ldots,s$.

The rows of the rectangular matrix $\left(\frac{\dd P_j}{\dd
b}(A)\right)$, corresponding to $j\in F$ (that is $j\neq
j_1,\ldots,j_s$), are linearly independent according to Lemma
\ref{lemMatrixFractions}. Besides, the rows of the rectangular
matrix $\left(\frac{\dd P_j}{\dd b}(A),\frac{\dd P_j}{\dd
a}(A)\right)$, corresponding to $j\in D$ (that is $j=j_l$ for some
$l$), have zeroes at all positions except $\frac{\dd P_j}{\dd
a_{j_l}}(A)$ and $\frac{\dd P_j}{\dd a_{j_l+1}}(A)$ and, moreover
$\frac{\dd P_j}{\dd a_{j_l}}(A)\neq 0$. It follows that the matrix
$\left(\frac{\dd P_j}{\dd b}(A),\frac{\dd P_j}{\dd a}(A)\right)$
of the form
\[
\begin{tikzpicture}[decoration=brace] \matrix (m) [matrix of
math nodes,left delimiter={(},right delimiter={)}] {
\ast & \cdots & \ast & \ast &\ast&\ast& \cdots &\ast&\ast& \cdots &&& \ast \\
\ast & \cdots & \ast & \ast &\ast&\ast& \cdots &\ast&\ast& \cdots &\ast&\ast& \ast \\
0 & \cdots & 0 &  0 & \ast & \ast & 0 & 0 & 0 & 0 & 0 & 0 & 0  \\
0 & \cdots & 0 &  0 & 0 & 0 & 0 & \ast & \ast & 0 & 0 & 0 & 0  \\
0 & \cdots & 0 &  0 & \cdots &\cdots&\cdots&\cdots&\cdots&\cdots&\cdots&\cdots & 0  \\
0 & \cdots & 0 &  0 & 0 & 0 & 0 & 0 & 0 & 0 & \ast & \ast & 0 \\
};

\draw[decorate,transform canvas={xshift=-1.5em},thick]
(m-6-1.west) -- node[left=2pt] {$D$} (m-3-1.west);
\draw[decorate,transform canvas={xshift=-1.5em},thick]
(m-2-1.west) -- node[left=2pt] {$F$} (m-1-1.west);

\draw[decorate,transform canvas={yshift=0.5em},thick]
(m-1-1.north) -- node[above=2pt] {$\dd P_j/\dd b$} (m-1-3.north);
\draw[decorate,transform canvas={yshift=0.5em},thick]
(m-1-4.north) -- node[above=2pt] {$\dd P_j/\dd a$} (m-1-13.north);

\draw (m-2-1.south west) -- (m-2-13.south east);

\draw (m-1-3.north east) -- (m-6-3.south east);

\draw[decorate,transform canvas={yshift=-0.5em},thick]
(m-6-6.south) -- node[below=2pt] {$\{j_1,j_1+1\}$} (m-6-5.south);

\draw[decorate,transform canvas={yshift=-0.5em},thick]
(m-6-9.south) -- node[below=2pt] {$\{j_2,j_2+1\}$} (m-6-8.south);

\draw[decorate,transform canvas={yshift=-0.5em},thick]
(m-6-12.south) -- node[below=2pt] {$\{j_s,j_s+1\}$}
(m-6-11.south);

%
%
\end{tikzpicture}
\]
has the maximal rank. Therefore $M_{\St_n,\lambda}$ is smooth at
all points.

Since the action of $T^n$ on $M_\Gamma=\Ro^{n+1}\times\Co^n$ is
locally standard and smooth submanifold $M_{\St_n,\lambda}$ is
preserved by the action, the induced action of $T^n$ on
$M_{\St_n,\lambda}$ is locally standard according to slice
theorem.
\end{proof}

For convenience we denote the orbit space $M_{\St_n,\lambda}/T^n$
by $Q_n$.

\begin{prop}
The map $\widetilde{\mu}\colon Q_n\to \ca{R}_n$ induced by $\mu$
is a homotopy equivalence.
\end{prop}

\begin{proof}
It is sufficient to prove that the preimage
$\widetilde{\mu}^{-1}(a)$ is contractible for any point $a\in
\ca{R}_n$. The condition on the spectrum yields the system of
equations
\begin{equation}\label{eqPreimageParameters}
|b_1|^2\prod_{i\neq 0,1}(a_i-\lambda_j)+\ldots+|b_1|^2\prod_{i\neq
0,n}(a_i-\lambda_j)=\prod_{i=0}^n(a_i-\lambda_j),\qquad
j=0,1,\ldots,n.
\end{equation}
according to Corollary \ref{corDetOfArrow}. Therefore, with the
diagonal elements $a_i$ and the eigenvalues $\lambda_i$ fixed, the
possible off-diagonal elements $b_i$ lie on the intersection of
Hermitian real quadrics of special type. By passing to the orbit
space we simply forget the arguments of the numbers $b_i$. Setting
$c_i=|b_i|^2$, we see that the parameters $c_i$ satisfy the system
of linear equations and conditions $c_i\geqslant 0$. Whenever this
set is non-empty, it is a convex polytope hence contractible.
\end{proof}

\begin{rem}\label{remMomAngCubes}
For generic points $a\in\ca{R}_n$ the preimage
$\widetilde{\mu}^{-1}(a)$ is a single point. In nongeneric points
the preimage $\widetilde{\mu}^{-1}(a)$ is a cube. It can be seen
that each pair of collided values $a_{j_l}=a_{j_l+1}$ produces an
interval in the preimage of $\widetilde{\mu}$. This interval is
parametrized by the barycentric coordinates
$|b_{j_l}|^2,|b_{j_l+1}|^2$ subject to the relation
$|b_{j_l}|^2+|b_{j_l+1}|^2=\const$ (note that if
$a_{j_l}=a_{j_l+1}$, then the expression
$|b_{j_l}|^2+|b_{j_l+1}|^2$ separates in all the
equations~\eqref{eqPreimageParameters}).

The total preimage $\mu^{-1}(a)$ is therefore homeomorphic to the
product of 3-spheres (a 3-sphere is the moment-angle manifold
corresponding to the interval, see \cite{BPnew} for a general
theory of moment-angle manifolds and complexes).
\end{rem}

To determine the homotopy type of the orbit space $Q_n\simeq
\ca{R}_n$, we need a description of the combinatorics of a
permutohedron (details could be found in many sources, e.g. in
\cite{Tomei}). We fix a finite set $[n]=\{1,\ldots,n\}$.

\begin{con}\label{consCombPermut}
Let $S=(S_1,\ldots,S_k)$ be an arbitrary linearly ordered
partition of the set $[n]=\{1,\ldots,n\}$ into nonempty subsets,
that is $S_i\cap S_j=\varnothing$ for $i\neq j$, and
$[n]=\bigcup_iS_i$. The set $P$ of all such partitions is
partially ordered: $S<S'$ if $S$ is an order preserving refinement
of $S'$. It is known that $P$ is isomorphic to the partially
ordered set of faces of the permutohedron $\Pe^{n-1}$. The
polytope itself corresponds to the maximal partition $S_1=[n]$.
The vertices correspond to ordered partitions of $[n]$ into
one-element subsets $(\{s_1\},\ldots,\{s_n\})$, which are actually
just the permutations $\tau\in\Sigma_n$, $\tau(i)=s_i$.

There is an edge between two vertices of a permutohedron if the
corresponding permutations differ by a transposition of $i$ and
$i+1$. As a corollary, we obtain a standard fact that the
1-skeleton of the permutohedron is the Cayley graph of the group
$\Sigma_n$ with the generators $(1,2), (2,3), \ldots, (n-1,n)$.

Let $F_S$ be the face of a permutohedron, corresponding to the
ordered partition $S=(S_1,\ldots,S_k)$. The polytope $F_S$ is
combinatorially isomorphic to the product of permutohedra
$\Pe^{|S_1|-1}\times\cdots\times\Pe^{|S_k|-1}$. If $|S_i|\leqslant
2$ for all $i$, then the corresponding face $F_S$ is a product of
intervals and points. We will call such faces \emph{cubical}. A
cubical face of $\Pe^{n-1}$ has dimension at most $[n/2]$.
\end{con}

\begin{rem}
We make a remark on another important fact. Consider the standard
convex realization of the permutohedron
\[
\Pe^{n-1}=\conv\{(x_{\sigma(1)},\ldots,x_{\sigma(n)})\mid
\sigma\in\Sigma_n\},
\]
where $x_1<x_2<\cdots<x_n$. The vertex
$(x_{\sigma(1)},\ldots,x_{\sigma(n)})$ in this convex realization
corresponds to the vertex $(\tau(1),\ldots,\tau(n))$ in the
combinatorial description, where $\tau=\sigma^{-1}$. For this
reason, there is a duality in the notation: it is more convenient
to encode the vertices by permutations $\sigma$ in geometrical
tasks, and by $\tau$ in combinatorial tasks.
\end{rem}

Let $\Sq_{n-1}$ be the cell complex, which consists of all cubical
faces of $(n-1)$-dimensional permutohedron. The complex
$\Sq_{n-1}$ is connected, since every edge of $\Pe^{n-1}$ is a
cubical face, therefore lies in $\Sq_{n-1}$.

\begin{prop}
The orbit space $Q_n = M_{\Gamma,\lambda}/T\simeq \ca{R}_n$ is
homotopy equivalent to~$\Sq_{n-1}$.
\end{prop}

\begin{proof}
We construct the intermediate simplicial complex $N$ which is
homotopy equivalent to both $\Sq_{n-1}$ and $\ca{R}_n$.

Let $N$ be the simplicial complex, obtained from $\Sq_{n-1}$ by
substituting each cubical face with vertices
$\{v_1,\ldots,v_{2^k}\}$ by a simplex on the same vertex set. The
natural map $N\to \Sq_{n-1}$, which is identical on vertices and
linear on each simplex, is a homotopy equivalence (see \cite{AB},
where the objects of this kind were studied in the theory of nerve
complexes).

Proposition \ref{propMomentImage} implies that
$\ca{R}_n=\mu(M_{\Gamma,\lambda})$ is the union of cubes
$\bigcup_{\tau\in \Sigma}I^n_{\tau}$, where
$I^n_{\tau}=I_{\tau(1)}\times\cdots\times I_{\tau(n)}$,
$I_j=[\lambda_{j-1},\lambda_j]$. All cubes $I^n_{\tau}$ are
convex. Hence the nerve $N'$ of the covering $\bigcup_{\tau\in
\Sigma}I^n_{\tau}=\ca{R}_n$ is homotopy equivalent to $\ca{R}_n$
by the nerve theorem.

Inspecting all coordinates, we can see that the cubes
$I^n_{\tau_1}$, $I^n_{\tau_2}$ intersect if and only if, for every
$i\in[n]$, there holds $|\tau_1(i)-\tau_2(i)|\leqslant 1$. This
means that
\begin{equation}\label{eqIntersectCubesCondition}
\tau_1\tau_2^{-1}=(i_1,i_1+1)(i_2,i_2+1)\ldots(i_s,i_s+1),\qquad
i_{l+q}>i_l+1,
\end{equation}
the product of independent transpositions, interchanging
neighboring elements. Therefore, the vertices $F_{\tau_1},
F_{\tau_2}$ of the permotohedron $\Pe^{n-1}$ lie in a cubical face
$F_S$ corresponding to the partition
\[
S=(\{\tau(1)\},\{\tau(2)\},\ldots,
\{\tau(i_1),\tau(i_1+1)\},\ldots,\{\tau(i_s),\tau(i_s+1)\},
\ldots,\{\tau(n)\}),
\]
where $\tau=\tau_1$ or $\tau_2$. More generally, let a family of
cubes $I^n_{\tau_i}, i=1,\ldots,l$ intersect in total. In this
case each pair $i<j$ determines its own product of transpositions
of the form \eqref{eqIntersectCubesCondition}. All permutations
$s_{1,j}=\tau_j\tau_1^{-1}$ have order two and commute (since
$s_{1,j}s_{1,i}^{-1}=\tau_j\tau_i^{-1}$ is of order two as well).
Therefore, there exists a common partition into 1- and 2-element
subsets, which governs all of these permutations. Again, all
vertices $F_{\tau_i}$ lie in the same cubical face of a
permutohedron. Thus we have $\Sq_{n-1}\simeq N =
N'\simeq\ca{R}_n$.
\end{proof}

\begin{ex}
For $n=3$, the orbit space $Q_3$ and the image of the moment map
are homotopy equivalent to $\Sq_2$. This complex is just the union
of cubical faces of the hexagon, which coincides with its
boundary: $Q\simeq S^1$. This can be seen from
Fig.\ref{pictNecklace}.

For $n=4$, the orbit space $Q_4$ is homotopy equivalent to
$\Sq_3$. This complex is shown on Fig.\ref{pictSq}. The complex
$\Sq_3$ is homotopy equivalent to a wedge of $7$ circles.
\end{ex}

\begin{figure}[h]
\begin{center}
\includegraphics[scale=0.3]{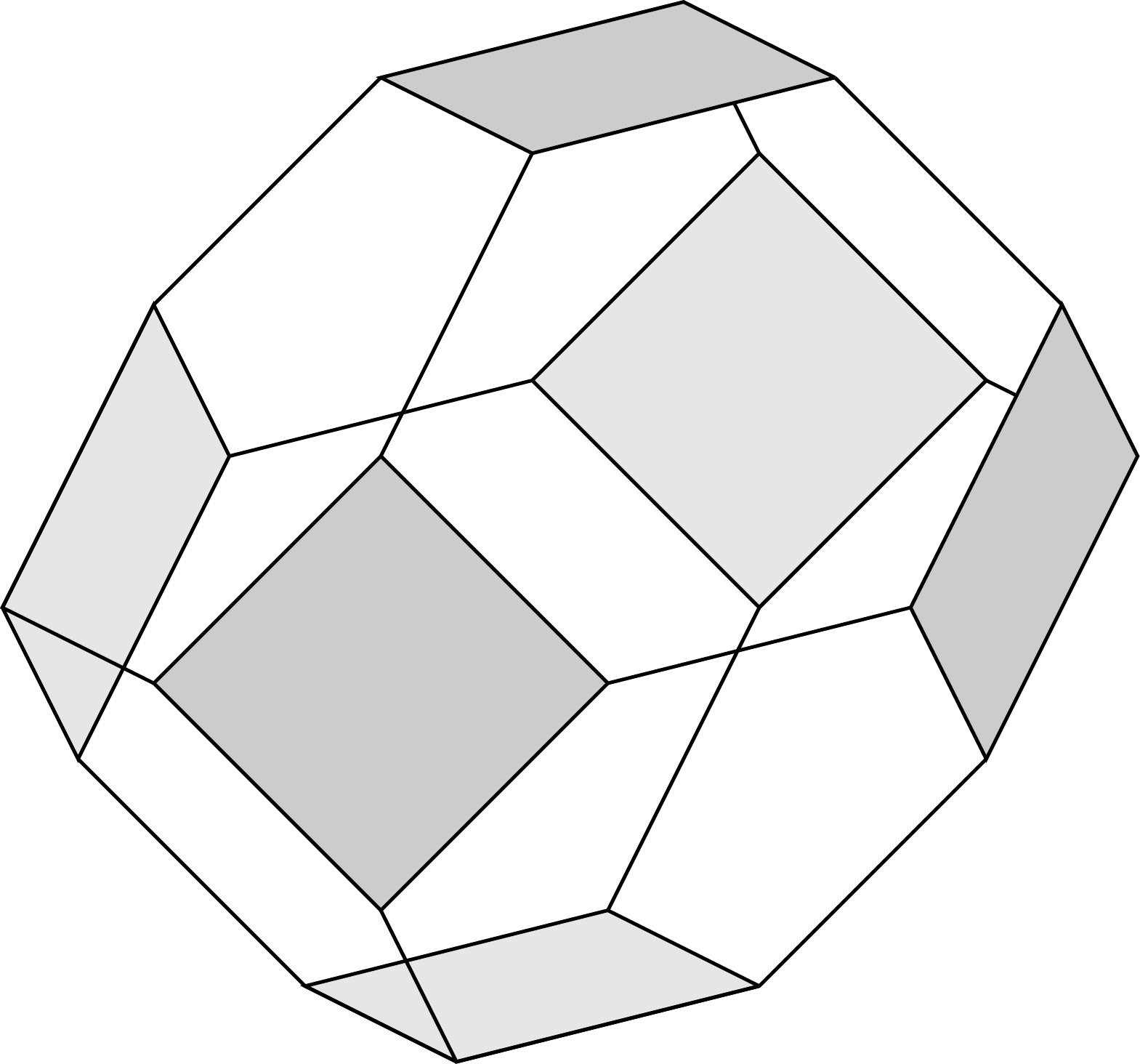}
\end{center}
\caption{The cubical complex $\Sq_3$}\label{pictSq}
\end{figure}

%
%
%
%
%
%

\section{Permutation action and the fundamental polytope}\label{secArrowMatricesBlock}

\begin{defin}
Let $P$ be a simple polytope $\dim P=n$, and $G_1,\ldots,G_s$ be a
collection of its codimension two faces. Let $\F_i', \F_i''$ be
the facets of $P$ such that $\F_i'\cap\F_i''=G_i$. If the sets
$\{\F_i',\F_i''\}$, $i=1,\ldots,s$, are disjoint, we call
$\{G_i\}$ a collection in general position.
\end{defin}

\begin{lem}\label{lemGeneralCut}
If $\{G_1,\ldots,G_s\}$ is a collection in general position, then
a simple combinatorial polytope, obtained by consecutive cutting
of these faces does not depend on the order of cuts.
\end{lem}

\begin{proof}
The statement easily follows from the consideration of the dual
simplicial sphere. Cutting a face of codimension two corresponds
to a stellar subdivision of an edge in a simplicial sphere. The
condition of general position implies that the subdivided edges do
not intersect. The independence of the result follows from the
definition of a stellar subdivision.
\end{proof}

Let $I^n=I_1\times\cdots\times I_n$ be the cube, $I_j=[-1;1]$. The
facets are indexed by the set $W=\{\pm1,\pm2,\ldots,\pm n\}$: the
element $\delta k$, $\delta=\pm1$, encodes the facet
\[
\F_{\delta k}=I_1\times\cdots\times
\stackrel{k}{\{\delta\}}\times\cdots\times I_n.
\]
The group $\Zt^n$ acts on the set of facets: for an element
$\epsilon=(\epsilon_1,\ldots,\epsilon_n)\in\Zt^n$,
$\epsilon=\pm1$, we have $\epsilon \delta k=\epsilon_k\delta k$.

Given the transposition of neighboring elements $\sigma_i=(i,
i+1)\in\Sigma_n$, $i=1,\ldots,n-1$, we denote by $F_{\sigma_i}$
the face of $I^n$ of the form
\[
F_{\sigma_i}=I_1\times\cdots\times \stackrel{i}{\{1\}}\times
\stackrel{i+1}{\{-1\}}\times\cdots\times I_n =
\F_i\cap\F_{-(i+1)}.
\]
Hence $F_{\sigma_i}$ is a face of codimension two.

\begin{defin}
Let $\B^n$ denote a simple polytope obtained from $I^n$ by cutting
all faces $F_{\sigma_i}$, $i=1,\ldots,n-1$.
\end{defin}

Note that the faces $F_{\sigma_i}$ may intersect, however the
result of the cutting is well defined according to Lemma
\ref{lemGeneralCut}.

\begin{figure}[h]
\begin{center}
\includegraphics[scale=0.4]{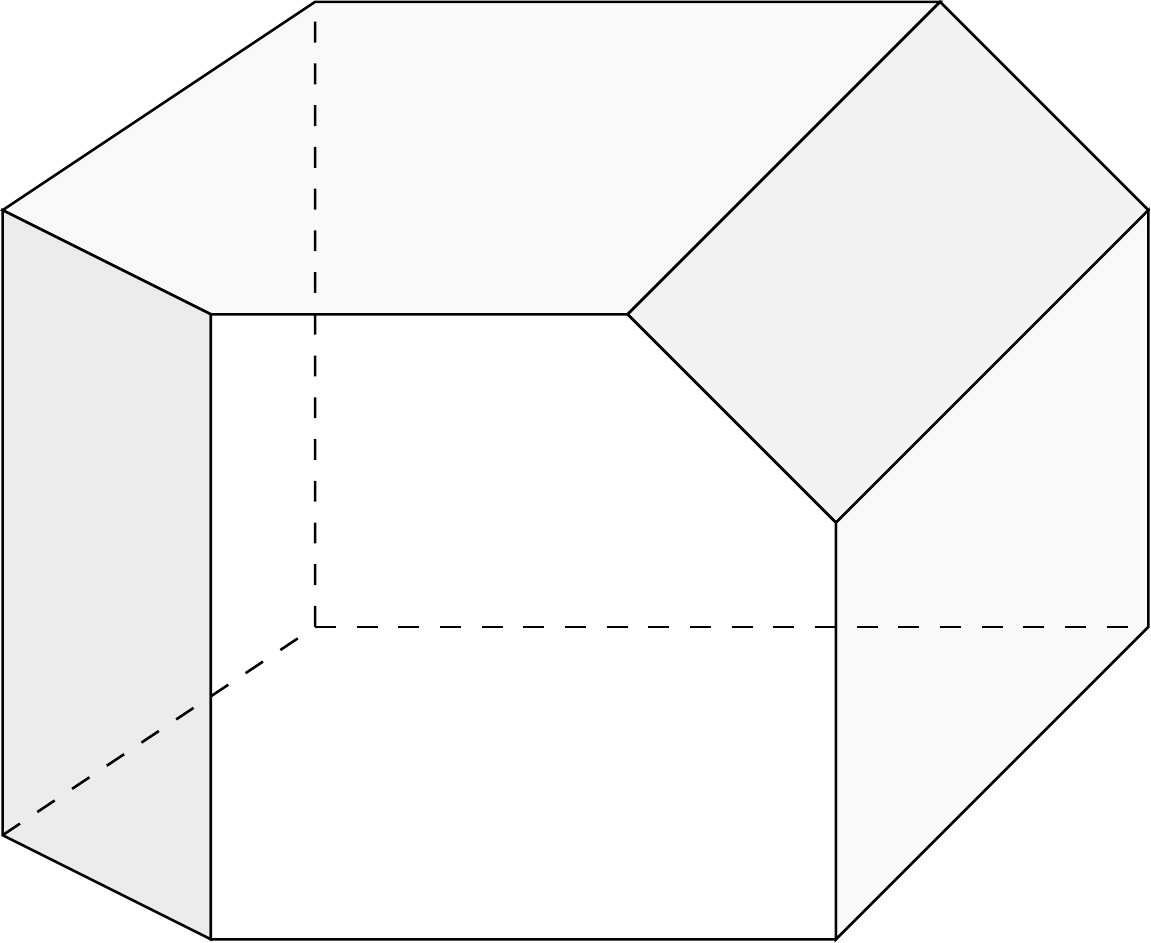}
\end{center}
\caption{The polytope $\B^3$}\label{pictCutCube}
\end{figure}

\begin{ex}
The polytope $\B^3$ is obtained from $I^3$ by cutting two skew
edges, see Fig.\ref{pictCutCube}. This polytope plays an important
role in toric topology, see \cite[Sect.4.9]{BPnew}, though it
emerged under different circumstances.
\end{ex}

The facet of $\B^n$ obtained by cutting $F_{\sigma_i}$ will be
denoted by $\F_{\sigma_i}$. The original facets of the cube remain
the facets of $\B^n$. They will be denoted by the same letters
$\F_{\delta k}$. In total, the polytope $\B^n$ has $3n-1$ facets.

Whenever a codimension two face is cut in a polytope $P^n$, the
resulting facet $F^{n-1}_{cut}$ has combinatorial type of the
product $I^1\times G^{n-2}$, for some polytope $G^{n-2}$. This can
be seen by considering the dual triangulation of a sphere:
whenever an edge $e$ is subdivided by a vertex $v$, the link of
$v$ is the suspension over the link of $e$ in the original
triangulation.

Every facet $F_{\sigma_i}$ of $\B^n$ has the form $I^1\times G_i$.
Let $\alpha_i\colon F_{\sigma_i}\to F_{\sigma_i}$ be the antipodal
map, that is the map which is constant on $G_i$ and antipodal on
the interval $I^1$.

\begin{con}
Note that the permutation group $\Sigma_n$ acts on a star graph.
This action induces the action of $\Sigma_n$ on the space of arrow
matrices which preserves the spectrum. Therefore there is an
action of $\Sigma_n$ on $M_{\St_n,\lambda}$ which commutes with
the torus action up to a natural action of $\Sigma_n$ on $T^n$.
Let $\Nt$ denote the semidirect product
$\Nt_n=T^n\rtimes\Sigma_n$, where $\Sigma_n$ acts on $T^n$ by
permuting the coordinates. The group $\Nt_n$ naturally arises as
the normalizer of the maximal torus $T^n$ in the Lie group $U(n)$.

We have an action of $\Nt_n$ on $M_{\St_n,\lambda}$. The orbit
space $Q^n=M_{\St_n,\lambda}/T^n$ carries the remaining action of
$\Nt/T^n\cong \Sigma_n$. The moment map image
$\ca{R}_n=\bigcup_{\sigma\in\Sigma_n}I^n_\sigma$ carries a natural
action of $\Sigma_n$, which permutes the coordinates (in
particular this action permutes cubes in the union). The map
$\tilde{\mu}\colon Q_n\to \ca{R}_n$ is $\Sigma_n$-equivariant, as
can be easily seen from its definition.

Similar considerations hold true in the real case. The finite
group $\Nt^\Ro_n=\Zt^n\rtimes\Sigma_n$ acts on
$M_{\St_n,\lambda}^{\Ro}$. Note that the group $\Nt^\Ro_n$
coincides with the Weyl group of type $B$.
\end{con}

\begin{prop}
The preimage of a single cube $I^n_\sigma$ of the set
$\ca{R}_n=\bigcup_{\sigma\in\Sigma_n}I^n_\sigma$ under the map
$\tilde{\mu}\colon Q_n\to \ca{R}_n$ is diffeomorphic to the
polytope $\B^n$. The map $\tilde{\mu}\colon
\tilde{\mu}^{-1}(I^n_\sigma)\to I^n_\sigma$ is the map $\B^n\to
I^n$ which blows down the cut faces.

The polytope $\B^n$ is the fundamental domain of the
$\Sigma_n$-action on $Q_n$, $\Nt_n$-action on $M_{\St_n,\lambda}$,
and $\Nt^\Ro_n$-action on $M_{\St_n,\lambda}^{\Ro}$.
\end{prop}

\begin{proof}
Without loss of generality, consider a single cube
$I^n_{\id}=I_1\times\cdots\times I_n$ corresponding to the trivial
permutation $\id\in\Sigma$. For the points in this cube we have
$a_1\leqslant a_2\leqslant\cdots\leqslant a_n$. As follows from
Section \ref{secArrowMatricesOrbits} (see remark
\ref{remMomAngCubes}) the preimage $\tilde{\mu}^{-1}(x)$ consists
of a single point for generic $x\in I^n_{\id}$. If a collision
$a_j=a_{j+1}$ occurs for a point $x$, this means that $x$ lies on
a codimension-two face $F_{\sigma_i}$ of a cube. It was noted in
remark~\ref{remMomAngCubes}, that in this case the preimage
$\tilde{\mu}^{-1}(x)$ is a cube of dimension equal to the number
of pairwise collisions. Therefore the preimage
$\tilde{\mu}^{-1}(I^n_{\id})$ is given by blowing up the cube at
the faces $F_{\sigma_1},\ldots,F_{\sigma_{n-1}}$.
\end{proof}

The orbit space $Q_n$ can be represented as the union of $n!$
copies of the polytope $\B^n$ attached along the cut faces. More
precisely, we have
\[
Q_n=\bigcup_{\sigma\in\Sigma_n}\B^n_\sigma/\sim,\quad
\B^n_\sigma\cong \B^n,
\]
The relation $\sim$ identifies the point $x\in\F_{\sigma_i}$ of
the polytope $\B^n_\sigma$ with the point
$\alpha_i(x)\in\F_{\sigma_i}$ of the polytope $\B^n_\tau$ whenever
$\sigma=\tau\sigma_i$. Recall that $\alpha_i$ is the antipodal
involution of the facet~$\F_{\sigma_i}$. One should not forget
about this involution: when passing from $\B^n_\sigma$ to
$\B^n_\tau$ the barycentric coordinates $|b_{i}|^2,|b_{i+1}^2|$ on
the blown up face interchange.

\begin{ex}\label{exCase3orbitSpace}
In case $n=3$ the space $Q_3$ is obtained by stacking 6 copies of
the polytope $\B^3$ shown on Fig.\ref{pictCutCube}. The result of
this stacking is shown on Fig.\ref{pictStackedCubes}. It can be
seen that $Q_3$ is a solid torus and its boundary is subdivided
into hexagons in a regular way. Note that the stacking does not
produce additional faces, so the picture should be smoothened at
stack points. This is somehow similar to the construction of
origami templates in the theory of toric origami manifolds (see
\cite{AMPZ}).
\end{ex}

\begin{figure}[h]
\begin{center}
\includegraphics[scale=0.2]{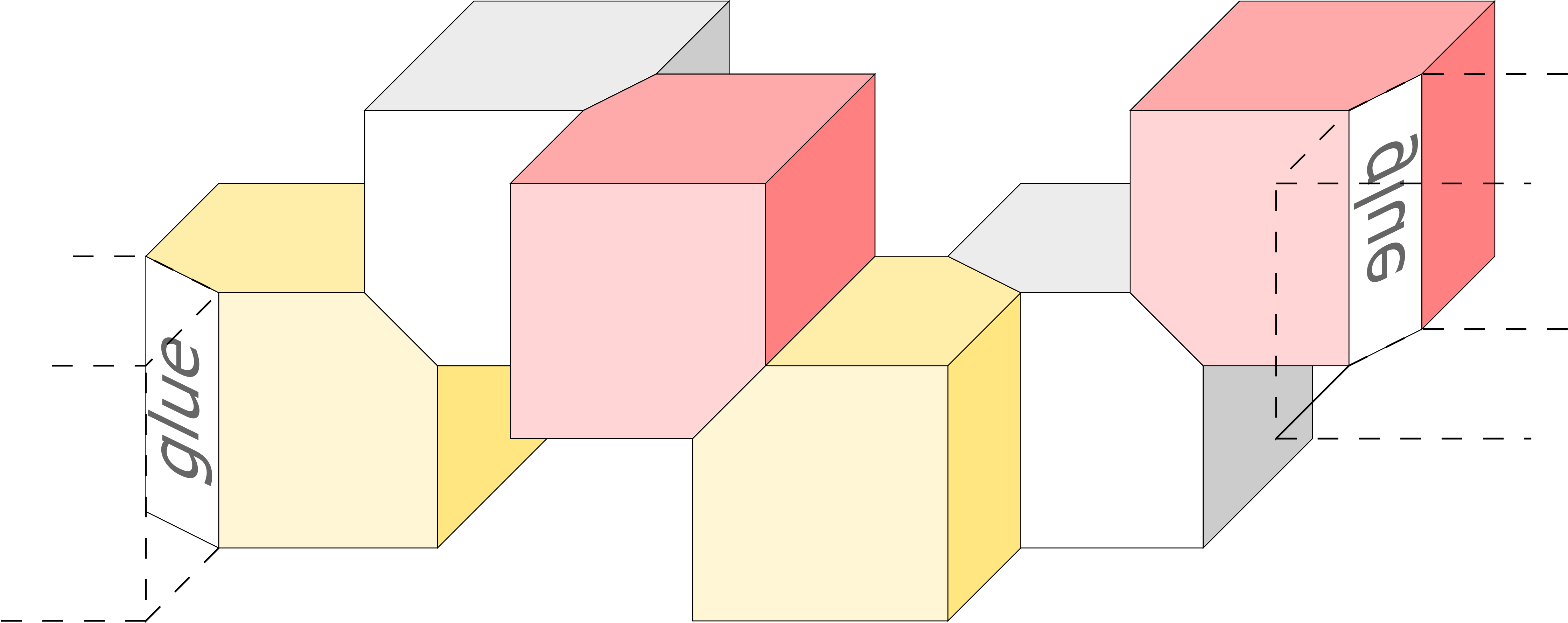}
\end{center}
\caption{Reconstruction of $Q_3$ by stacking 6 copies of
$\B^3$.}\label{pictStackedCubes}
\end{figure}

The manifold $M_{\St_n,\lambda}^\Ro$ is a small cover over
$Q_n=M_{\St_n,\lambda}^\Ro/\Zt^n$ in which all stabilizers of the
action $\Zt^n\circlearrowright M_{\St_n,\lambda}^\Ro$ are the
coordinate subgroups of $\Zt^n$. The constructed cellular
subdivision on $Q_n$ can be lifted to $M_{\St_n,\lambda}^\Ro$, and
we have the following statement.

\begin{thm}
The manifold $M_{\St_n,\lambda}^\Ro$ is canonically subdivided
into $2^nn!$ cells combinatorially isomorphic to $\B^n$. More
precisely,
\[
M_{\St_n,\lambda}^\Ro=\bigcup_{\sigma\in\Sigma_n,
\epsilon\in\Zt^n}\B^n_{\sigma,\epsilon}/\sim,\quad
\B^n_{\sigma,\epsilon}\cong \B^n,
\]
where the equivalence relation $\sim$ is generated by the
relations
\begin{itemize}
\item A point $x\in\F_{\sigma_i}\subset \B^n_{\sigma,\epsilon}$
is identified with the point
$\alpha_i(x)\in\F_{\sigma_i}\subset\B^n_{\tau,\epsilon}$ if
$\sigma\tau^{-1}=\sigma_i$.
\item A point $x\in\F_{\delta k}\subset\B^n_{\sigma,\epsilon_1}$ is
identified with the point $x\in\F_{\delta
k}\subset\B^n_{\sigma,\epsilon_2}$ if
$\epsilon_1\epsilon_2^{-1}=(+1,\ldots,+1,\stackrel{k}{-1},+1,\ldots,+1)$.
\end{itemize}
Each codimension $s$ cell of the cellular structure on
$M_{\St_n,\lambda}^\Ro$ lies in exactly $2^{n-s}$ top-dimensional
cells.
\end{thm}

\section{Combinatorics of strata}\label{secTreeMatricesCombinatorics}

In this section $\Gamma$ is an arbitrary graph on the set $V$,
$|V|=n$, $\Delta=(\Gamma,a,b)$ is a labeled graph, $A_\Delta$ its
corresponding Hermitian matrix, and $M_{\Gamma,\lambda}$ is the
space of $\Gamma$-shaped matrices with the simple spectrum
$\lambda$.

The orbit space $M_{\Gamma,\lambda}/T^n$ is stratified by
dimensions of torus orbits. Suppose an orbit $[A]\in
M_{\Gamma,\lambda}/T^n$ is represented by a matrix $A$ such that
$b_e=0$, $e\in W$ for some set $W$ of edges. Discarding the edges
of $W$ from the graph $\Gamma$, we obtain a new graph
$\tilde{\Gamma}$ on the same vertex set $V$. Let $k$ be the number
of connected components of $\tilde{\Gamma}$. It can be seen that
the dimension of the orbit $T^nA$ is $n-k$. The set of matrices
having zeroes at positions $e\in W$ in general may be
disconnected, since the spectrum $\lambda$ can distribute between
blocks of a matrix in different ways. These considerations
motivate the following definition.

\begin{defin}
Let $\Gamma$ be a graph on a vertex set $V$, $|V|=n$. Let
$\tilde{\Gamma}\subset \Gamma$ be a subgraph on $V$, and
$V_1,\ldots,V_k$ are the vertex sets of the connected components
of $\tilde{\Gamma}$. Consider $\ca{V}=\{V_1,\ldots,V_k\}$, the
unordered partition of $V$. We say that two bijections
$p_1,p_2\colon V\to[n]$ are \emph{equivalent with respect to}
$\ca{V}$ (or simply $\ca{V}$-\emph{equivalent}), if they differ by
permutations within each of the subsets $V_i$, that is
\[
p_1=p_2\cdot\sigma,\quad \sigma\in
\Sigma_{V_1}\times\cdots\times\Sigma_{V_k}\subseteq\Sigma_V.
\]
The class of $\ca{V}$-equivalent bijections will be called a
\emph{cluster}, submitted to the partition $\ca{V}$. Let
$\ca{P}_\Gamma$ be the set of all clusters submitted to partitions
of subgraphs $\tilde{\Gamma}\subset\Gamma$ into connected
components. Now we define a partial order on $\ca{P}_\Gamma$. Note
that the inclusion of subgraphs $\tilde{\Gamma}'\subset
\tilde{\Gamma}$ implies that the partition $\tilde{\ca{V}}'$ is a
refinement of $\tilde{\ca{V}}$. We say that a cluster $[p']$
submitted to $\tilde{\ca{V}}'$ is less than a cluster $[p]$
submitted to $\tilde{\ca{V}}$ if $p$ and $p'$ are
$\ca{V}$-equivalent. In other words, $p'<p$, if $p'$ is a
refinement of $p$. The poset $\ca{P}_\Gamma$ is called a
\emph{cluster-permutohedron} corresponding to $\Gamma$.
\end{defin}

\begin{rem}
A cluster-permutohedron is a graded poset: the rank of a cluster
submitted to a subgraph $\tilde{\Gamma}$ equals the rank of this
subgraph, i.e. the number of edges in its spanning forest.
Therefore, the rank is equal to the number of vertices of
$\tilde{\Gamma}$ minus the number of its connected components.
\end{rem}

Note that the poset $\ca{P}_\Gamma$ has a unique maximal element
of rank $n-1$, represented by the cluster submitted to the whole
graph $\Gamma$. The elements in this cluster can be permuted
arbitrarily. A smallest element of the poset $\ca{P}_\Gamma$
corresponds to a partition of the vertex set into $n$ singletons:
such clusters are encoded by all possible permutations $p\colon
V\to[n]$. This means $\ca{P}_\Gamma$ has exactly $n!$ atoms for
any $\Gamma$.

\begin{ex}
Let $\Gamma$ be a path graph with edges $(1,2),
(2,3),\ldots,(n-1,n)$. Its subgraph is represented by a sequence
of path graphs, and the corresponding partition of the vertex set
has the form
\[
\ca{V}=\{\{1,\ldots,s_1\},\{s_1+1,\ldots,s_2\},\ldots,\{s_k+1,\ldots,n\}\},
\]
where $1\leqslant s_1<s_2<\cdots<s_k<n$. A cluster submitted to
this partition has the form
\[
(\{\sigma(1),\ldots,\sigma(s_1)\},\{\sigma(s_1+1),\ldots,\sigma(s_2)\},
\ldots,\{\sigma(s_k+1),\ldots,\sigma(n)\}),
\]
which is a linearly ordered partition of $V=[n]$. Therefore, the
cluster-permutohedron in this case coincides with the poset of
faces of a permutohedron, see Construction \ref{consCombPermut}.
\end{ex}

\begin{ex}
Let $\Gamma$ be a simple cycle on $[n]$, that is the graph with
edges $(1,2)$, $(2,3)$, $\ldots$, $(n-1,n)$, $(n,1)$. In this case
clusters submitted to partitions into connected components are the
cyclically ordered partitions of $[n]$. The poset of such
cyclically ordered partitions is called \emph{cyclopermutohedron};
it was introduced and studied in the works of Panina,
see~\cite{Panina}.
\end{ex}

We formulate several basic properties of general
cluster-permutohedra.

\begin{prop}\label{propIdealInClusterPerm}
Let $p\in \ca{P}_\Gamma$ be an element of the
cluster-permutohedron, submitted to a partition
$\ca{V}=\{V_1,\ldots,V_k\}$ of the vertex set of $\Gamma$, and let
$\Gamma_i$ be the full subgraph of $\Gamma$ on the set $V_i$. Then
the lower order ideal
\[
(\ca{P}_\Gamma)_{\leqslant p}=\{q\in \ca{P}_\Gamma\mid q<p\}
\]
is isomorphic to the direct product of posets
\[
\ca{P}_{\Gamma_1}\times \cdots \times \ca{P}_{\Gamma_k}
\]
\end{prop}

The proof is straightforward from the construction of a partial
order. Recall the basic definitions from the theory of posets.

\begin{defin}
The poset $S$ is called \emph{simplicial} if it has a unique
minimal element $\hat{0}\subset S$, and for each $I\in S$ the
lower order ideal $S_{\leqslant I}=\{J\in S\mid J\leqslant I\}$ is
isomorphic to a boolean lattice. The elements of a simplicial
poset are called \emph{simplices}. If every two simplices $I,J\in
S$ have unique common lower bound, then $S$ is called a
\emph{simplicial complex}.
\end{defin}

A poset $S$ will be called \emph{simple} (or dually simplicial) if
$S^*$ is a simplicial poset. Here $S^*$ is the set $S$ with the
order reversed. By the definition of a simplicial complex, each
simplex is uniquely determined by its set of vertices. A
simplicial complex $S$ is called flag, if, whenever a collection
$\sigma$ of vertices is pairwise connected by edges, then $\sigma$
is a simplex of $K$.

\begin{prop}
If $\Gamma$ is a tree, then the poset $\ca{P}_\Gamma$ is simple.
Its dual $\ca{P}_\Gamma^*$ is a flag simplicial complex.
\end{prop}

\begin{proof}
Let an element $p\in\ca{P}_\Gamma$ be submitted to a partition
$\ca{V}=\{V_1,\ldots,V_k\}$. This partition defines a forest
$\tilde{\Gamma}\subset\Gamma$ uniquely. Let
$\{e_1,\ldots,e_{k-1}\}$ be the set of edges of $\Gamma$ which do
not lie in $\tilde{\Gamma}$. The upper order ideal $\{q\in
\ca{P}_\Gamma\mid q\geqslant p\}$ is isomorphic to the boolean
lattice of subsets of the set $\{e_1,\ldots,e_{k-1}\}$. The
condition on simplicial complex and the flag property hold for
similar reasons.
\end{proof}

\begin{ex}
Let $\Gamma=\St_n$ be a star graph with edges
$(0,1),(0,2),\ldots,(0,n)$. Its cluster-permutohedron has the
property that each of its lower order ideals is again a
cluster-permutohedron of a star graph. Indeed, each subgraph of
$\St_n$ is a disjoint union of a discrete set of vertices
$\{i_1,\ldots,i_s\}\subset\{1,\ldots,n\}$ and a star graph on the
remaining vertices.
\end{ex}

This example is particularly interesting in connection with the
results of Section \ref{secArrowMatricesOrbits}. The space
$M_{\St_n,\lambda}$ is a smooth manifold with half-dimensional
torus action. Therefore the orbit space $Q_n=M_{\St_n,\lambda}/T$
is a manifold with corners. The poset of faces of this manifold
with corners is isomorphic to the cluster-permutohedron
$\ca{P}_{\St_n}$ according to the arguments from the beginning of
this section.

We consider the case $n=3$, the star with $3$ rays, in detail.
First note that the star graph $\St_2$ is just a simple path on
$3$ vertices, hence its cluster-permutohedron $\ca{P}_\Gamma$ is
the poset of faces of a hexagon, according to Example
\ref{exTridiagonal}. All 2-dimensional faces of $Q_3$ are the
hexagons. Indeed, each of these faces is an orbit space of the
manifold of tridiagonal $3\times3$-matrices. This orbit space is
the permutohedron of dimension $2$, that is the hexagon.

As shown in Section \ref{secArrowMatricesBlock}, the space $Q_3$
is a solid torus, whose boundary is subdivided into hexagons in a
regular simple fashion: each vertex is contained in exactly $3$
hexagons. The combinatorics of this hexagonal subdivision can be
described in terms of the cluster-permutohedron, see.
Fig.\ref{pictHexTorusDetailed}.

We remark that the cell complex $\dd Q_3$ is a nanotube with
chiral vectors $(2,2)$ and $(4,-2)$, according to terminology
adopted in discrete geometry and chemistry \cite{BE}.

\begin{figure}
\begin{center}
\includegraphics[scale=0.4]{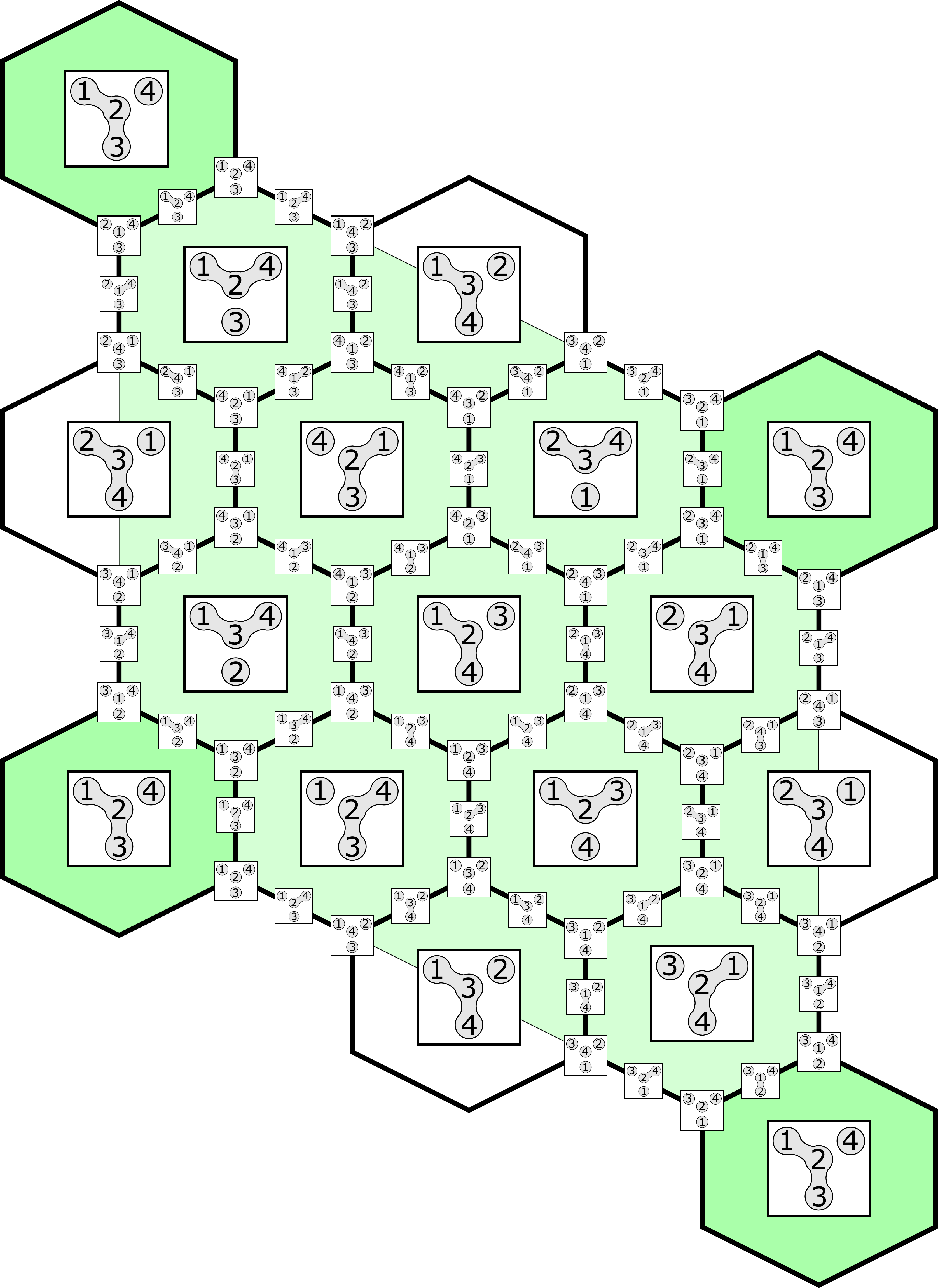}
\end{center}
\caption{The combinatorics of the cluster-permutohedron for the
star graph with 3 rays.}\label{pictHexTorusDetailed}
\end{figure}


\section{Arrow matrices $4\times
4$.}\label{secArrowMatricesTopology}

We apply the results of \cite{Ay1,Ay2,AMPZ} to describe the
cohomology and equivariant cohomology rings of the manifold
$M_{\St_3,\lambda}$.

Recall the necessary definitions. Let $Q$ be a manifold with
corners, $\dim Q=n$. Assume that every vertex of $Q$ lies in
exactly $n$ facets (such manifolds with corners were called
\emph{nice} in \cite{MasPan}, otherwise \emph{manifolds with
faces}). We call $Q$ an \emph{almost homological polytope} (or
simply an \emph{almost polytope}), if all its proper faces are
acyclic. If, moreover, $Q$ itself is acyclic, $Q$ is called a
\emph{homology polytope}.

Let $\ca{P}_Q$ denote a poset of faces of $Q$. The poset
$K_Q=\ca{P}_Q^*$ with reversed order is simplicial for a nice
manifold with corners. If $Q$ is an almost polytope, then the
simplicial poset $K_Q$ is a homology manifold (see the definition
below).

We recall the basic definitions needed for our considerations. The
definitions are given for simplicial complexes, since general
simplicial posets do not appear in our examples.

\begin{defin}
Let $K$ be a simplicial complex on a set $[m]$, and $I\in K$ be a
simplex. The \emph{link} of $I$ is a simplicial complex
$\link_KI=\{J\subseteq[m]\setminus I\mid I\cap J \in K\}$. In
particular, we have $\link_K\varnothing=K$. The complex $K$ is
called \emph{pure} if all its maximal-by-inclusion simplices have
equal dimensions. A pure simplicial complex $K$ of dimension $n-1$
is called a \emph{homology manifold} whenever, for any nonempty
simplex $I\in K$, $I\neq\varnothing$, the complex $\link_KI$ has
the same homology as a sphere of the corresponding dimension:
\begin{equation}\label{eqAcyclicityCond}
\Hr_j(\link_KI;\Zo)=\begin{cases} \Zo,\mbox{ if }j=n-1-|I|;\\
0,\mbox{ otherwise.}
\end{cases}
\end{equation}
The manifold $K$ is called \emph{orientable} if its geometric
realization possess an orienting cycle. The complex $K$ is called
a \emph{homology sphere} if the condition \eqref{eqAcyclicityCond}
holds for $I=\varnothing$, $\link_K\varnothing=K$ as well.
\end{defin}

\begin{rem}
In all the definitions, the coefficient ring should be specified.
If the coefficient ring is omitted, it is assumed that the
coefficients are in $\Zo$.
\end{rem}

\begin{defin}
Let $R$ be either a field or the ring $\Zo$, and
$R[m]=R[v_1,\ldots,v_m]$ be the polynomial ring on $m$ generators,
$\deg v_i=2$. The Stanley--Reisner ring of $K$ is the quotient
ring
\[
R[K]=R[m]/(v_{i_1}v_{i_2}\cdots v_{i_k}\mid
\{i_1,\ldots,i_k\}\notin K).
\]
The ring $R[K]$ has a natural structure of $R[m]$-module.
\end{defin}

\begin{defin}
Let $K$ be a pure simplicial complex of dimension $n-1$. A
function $\lambda\colon [m]\to R^n$ is called a
\emph{characteristic function} for $K$, if for each maximal
simplex $I=\{i_1,\ldots,i_n\}$, the collection
$\{\lambda(i_1),\ldots,\lambda(i_n)\}$ is a basis of a free module
$R^n$.
\end{defin}

Each function $\lambda\colon [m]\to R^n$ determines a sequence of
degree two elements in the ring $R[K]$ as described below (abusing
the terminology we call such elements linear, since all rings are
generated in degree two, and there are no odd components at all).
Let $\lambda(i)=(\lambda_{i,1},\ldots,\lambda_{i,n})\in R^n$,
$i\in[m]$. Consider
$\theta_j=\lambda_{1,j}v_1+\cdots+\lambda_{m,j}v_m\in R[K]$,
$j=1,\ldots,n$. Let $\Theta$ be the ideal in $R[K]$ generated by
$\theta_1,\ldots,\theta_n$.

\begin{lem}[see e.g. \cite{BPnew}]
Let $R$ be a field. Then the set of linear elements
$\theta_1,\ldots,\theta_n$ is a linear system of parameters in
$R[K]$ if and only if $\lambda\colon [m]\to R^n$ is a
characteristic function.
\end{lem}

In the following, $\lambda\colon[m]\to R^n$ always denotes a
characteristic function.

\begin{ex}
Let $c\colon [m]\to [n]$ be a proper coloring of the vertices of
$K$, that is a map which takes different values at the endpoints
of any edge of $K$. For such coloring there is an associated
characteristic function $\lambda_c\colon[m]\to R^n$,
$\lambda_c(i)=e_{c(i)}$, where $e_1,\ldots,e_n$ is the fixed basis
of $R^n$. For characteristic function $\lambda_c$, the elements of
the linear system of parameters have the form
\[
\theta_j=\sum_{i, c(i)=j}v_i\in R[K]_2.
\]
Characteristic functions of this kind will be called
\emph{chromatic}.
\end{ex}

\begin{prop}[Stanley--Reisner \cite{Reis,St}, Schenzel \cite{Sch}]
If $K$ is a homology sphere, then $R[K]$ is a Cohen--Macaulay
ring. In this case $\theta_1,\ldots,\theta_n$ is a regular
sequence in $R[K]$, which means that $R[K]$ is a free module over
the subring $R[\theta_1,\ldots,\theta_n]$. If $K$ is a homology
manifold, then $R[K]$ is a Buchsbaum ring (which means that
$\theta_1,\ldots,\theta_n$ is a weak regular sequence.
\end{prop}

Recall the basic combinatorial characteristics of a simplicial
complex. Let $f_j$ denote the number of $j$-dimensional simplices
of $K$ for $j=-1,0,\ldots,n-1$, in particular, $f_{-1}=1$ (the
empty simplex has formal dimension $-1$). $h$-numbers of $K$ are
defined from the relation:
\begin{equation}\label{eqHvecDefin}
\sum_{j=0}^nh_jt^{n-j}=\sum_{j=0}^nf_{j-1}(t-1)^{n-j},
\end{equation}
where $t$ is a formal variable. Let $\br_j(K)=\dim \Hr_j(K)$ be
the reduced Betti number of $K$. \emph{$h'$- and $h''$-numbers} of
$K$ are defined by the relations
\begin{equation}\label{eqDefHprime}
h_j'=h_j+{n\choose
j}\left(\sum_{s=1}^{j-1}(-1)^{j-s-1}\br_{s-1}(K)\right)\mbox{ for
} 0\leqslant j\leqslant n;
\end{equation}
\begin{equation}\label{eqDefHtwoprimes}
h_j'' = h_j'-{n\choose j}\br_{j-1}(K) = h_j+{n\choose
j}\left(\sum_{s=1}^{j}(-1)^{j-s-1}\br_{s-1}(K)\right)
\end{equation}
for $0\leqslant j\leqslant n-1$, and $h''_n=h'_n$. Summation over
an empty set is assumed zero.

\begin{prop}[Stanley--Reisner \cite{Reis,St}, Schenzel \cite{Sch}]
For each pure simplicial complex $K$ of dimension $n-1$ there
holds
\[
\Hilb(R[K];t)=\dfrac{h_0+h_1t^2+\cdots+h_nt^n}{(1-t^2)^n}.
\]
For homology sphere $K$ there holds
$\Hilb(R[K]/\Theta;t)=\sum_ih_it^{2i}$. For homology manifold $K$
there holds $\Hilb(R[K]/\Theta;t)=\sum_ih'_it^{2i}$.
\end{prop}

\begin{prop}[Novik--Swartz \cite{NS,NSgor}]
Let $K$ be a connected orientable homology manifold of dimension
$n-1$. The $2j$-th graded component of the module $R[K]/\Theta$
contains a vector subspace $(I_{NS})_{2j}\cong{n\choose
j}\Hr^{j-1}(K;R)$, which is a trivial $R[m]$-submodule (i.e.
$R[m]_+(I_{NS})_{2j}=0$). Let
$I_{NS}=\bigoplus_{j=0}^{n-1}(I_{NS})_{2j}$ be the sum of all
submodules except for the top one. Then the quotient module
$R[K]/\Theta/I_{NS}$ is a Poincare duality algebra, and there
holds $\Hilb(R[K]/\Theta/I_{NS};t)=\sum_ih''_it^{2i}$.
\end{prop}

This theorem in particular implies the generalized
Dehn--Sommerville relations for manifolds: $h''_j=h''_{n-j}$
(these relations, however, can be proved without using hard
algebraic machinery).

Let $\Lambda^*R^n$ denote the exterior algebra over $R^n$. For a
characteristic function $\lambda\colon [m]\to R^n$ and an oriented
simplex $I=\{i_1,\ldots,i_s\}\in K$ consider the nonzero skew-form
\[
\lambda_I=\lambda(i_1)\wedge\cdots\wedge\lambda(i_s)\in
\Lambda^sR^n.
\]

\begin{prop}[Ayzenberg, \cite{Ay1,Ay3}]\label{propMyRelations}
Let $K$ be a connected orientable homology manifold, $\dim K=n-1$,
and $\theta_1,\ldots,\theta_n$ be a linear system of parameters
corresponding to a characteristic function $\lambda$. Let $R$ be
either $\Qo$ or $\Zo$. The $2j$-th graded component of the algebra
$R[K]/\Theta$ is additively generated by the elements
\[
v_I=v_{i_1}\cdots v_{i_j},\quad I=\{i_1,\ldots,i_j\}\in K,
\]
and we have the following.

(1) All additive relations on the elements $v_I$ in the module
$R[K]/\Theta$ have the form
\begin{equation}\label{eqRelsSR}
\sum_{I\in K,|I|=j}\langle\omega,\lambda_I\rangle \sigma(I) v_I,
\end{equation}
where $\mu$ runs over $(\Lambda^kR^n)^*$, and $\sigma$ runs over
the vector space of simplicial $(j-1)$-coboundaries of the complex
$K$: $\sigma\in \Ca^{j-1}(K;R)$, $\sigma=d\tau$.

(2) All additive relations on the elements $v_I$ in the module
$R[K]/\Theta/I_{NS}$ for $j<n$ have the form \eqref{eqRelsSR},
where $\sigma$ runs over the space of simplicial cocycles:
$\sigma\in \Ca^{j-1}(K;R)$, $d\sigma=0$.
\end{prop}

In particular, this statement gives an explicit formula for the
generators of Novik--Swartz ideal $I_{NS}\subset R[K]/\Theta$. In
the work \cite{AM} we called the relations \eqref{eqRelsSR} the
Minkowski type relations by analogy with the terminology adopted
in toric geometry. It was shown that these relations have simple
geometrical explanation, coming from the theory of
multi-polytopes.

The theory sketched above can be used to describe the
cohomological structure of manifolds with half-dimensional torus
action. Let $X$ be a $2n$-manifold, and let the compact torus
$T^n$ act on $X$ in a locally standard way. In this case the orbit
space $Q=X/T$ is a nice manifold with corners. Let $K_Q$ be the
simplicial poset dual to simple poset of faces of $Q$. Let $[m]$
be the vertex set of $K_Q$, and therefore, $\F_1,\ldots,\F_m$ be
the set of facets of $Q$.

The preimage of $\F_i\subset Q$ under the map $p\colon X\to Q$ is
a submanifold $X_i\subset X$ of codimension $2$, which is called a
\emph{characteristic submanifold}. The cohomology class dual to
$[X_i]\in H_{2n-2}(X)$ is denoted by $v_i\in H^2(X;\Zo)$ (one
should orient $X_i$ somehow to make things well-defined). For a
point $x$ lying in the interior of $\F_i$, the stabilizer of the
action is a one-dimensional subgroup. It has the form
$\lambda(i)(S^1)$, where $\lambda(i)\in \Hom(S^1,T^n)\cong\Zo^n$.
Since the action is locally standard, the map $\lambda$ is a
characteristic function on $K_Q$.

If $Q$ is a homology polytope, then $K_Q$ is a homology sphere. If
$Q$ is an almost polytope, then $K_Q$ is a homology manifold. The
proper faces of an almost polytope $Q$ determine the homological
cell subdivision of the boundary $\dd Q$, which is dual to $K_Q$,
see details in \cite{Ay0}.

\begin{prop}[Masuda--Panov \cite{MasPan}]
If $Q$ is a homology polytope, then
\[
H^*_T(X;\Zo)\cong \Zo[K_Q];\qquad H^*(X;\Zo)\cong \Zo[K_Q]/\Theta;
\]
\[
H^{2i+1}(X;\Zo)=0;\qquad H^{2i}(X;\Zo)\cong\Zo^{h_i(K_Q)}.
\]
\end{prop}

\begin{prop}[Ayzenberg--Masuda--Park--Zeng \cite{AMPZ}]\label{propEquivCohomAlmPoly}
If $Q$ is an orientable connected almost polytope and the
projection map $p\colon X\to Q$ admits a section, then
\[
H^*_T(X;R)\cong R[K_Q]\oplus H^*(Q;R),
\]
(the units are identified in the direct sum of the rings).
\end{prop}

\begin{prop}[Ayzenberg \cite{Ay2}]\label{propCohomAlmPoly}
Assume $Q$ is an orientable connected almost polytope and the
projection map $p\colon X\to Q$ admits a section.

(1) Let $A^*(X;R)$ be the subring in $H^*(X;R)$ generated by the
classes $v_i$ of charcteristic submanifolds. There exists a
sequence of epimorphisms
\[
R[K_Q]/\Theta\twoheadrightarrow A^*(X;R) \twoheadrightarrow
R[K_Q]/\Theta/I_{NS}.
\]
The component $A^{2j}(X;R)$ is additively generated by classes
$v_I$, $I\in K_Q$, $|I|=j$. The relations on these classes in
$A^{2j}(X;R)$ for $j<n$ have the form \eqref{eqRelsSR}, where
$\sigma\in \Ca^{j-1}(K;R)\cong \Ca_{n-j}(\dd Q;R)$ runs over all
cellular chains, which vanish in $H_{n-j}(Q;R)$.

(2) The submodule $A^+=\bigoplus_{j>0}A^{2j}(X;R)$ is an ideal in
$H^*(X;R)$. There is an isomorphism of graded rings
\[
H^*(X)/A^+\cong \left(\bigoplus_{i<j}H^i(Q,\dd Q)\otimes
H^j(T^n)\right)\oplus \left(\bigoplus_{i\geqslant j}H^i(Q)\otimes
H^j(T^n)\right).
\]
All nontrivial products on the right hand part are given by cup
products in cohomology and relative cohomology.
\end{prop}

Now we apply this technique to 6-dimensional manifold
$M_{\St_3,\lambda}$ of isospectral arrow $4\times 4$-matrices.
According to the results of Sections \ref{secArrowMatricesOrbits}
and \ref{secArrowMatricesBlock}, the orbit space
$Q_3=M_{\St_3,\lambda}/T^3$ is a manifold with corners,
homeomorphic to $D^2\times S^1$, and its boundary is subdivided
into hexagons as shown on Fig.\ref{pictTorusAndDual}. Therefore
$\Qs$ is an almost polytope. Note that the map $p\colon
M_{\St_3,\lambda}\to Q_3$ admits a section. Indeed, $Q_3$ may be
identified with the space of arrow matrices, which have
nonnegative off-diagonal elements, which is the natural subset of
$Q_3$.

\begin{figure}
\begin{center}
\includegraphics[scale=0.25]{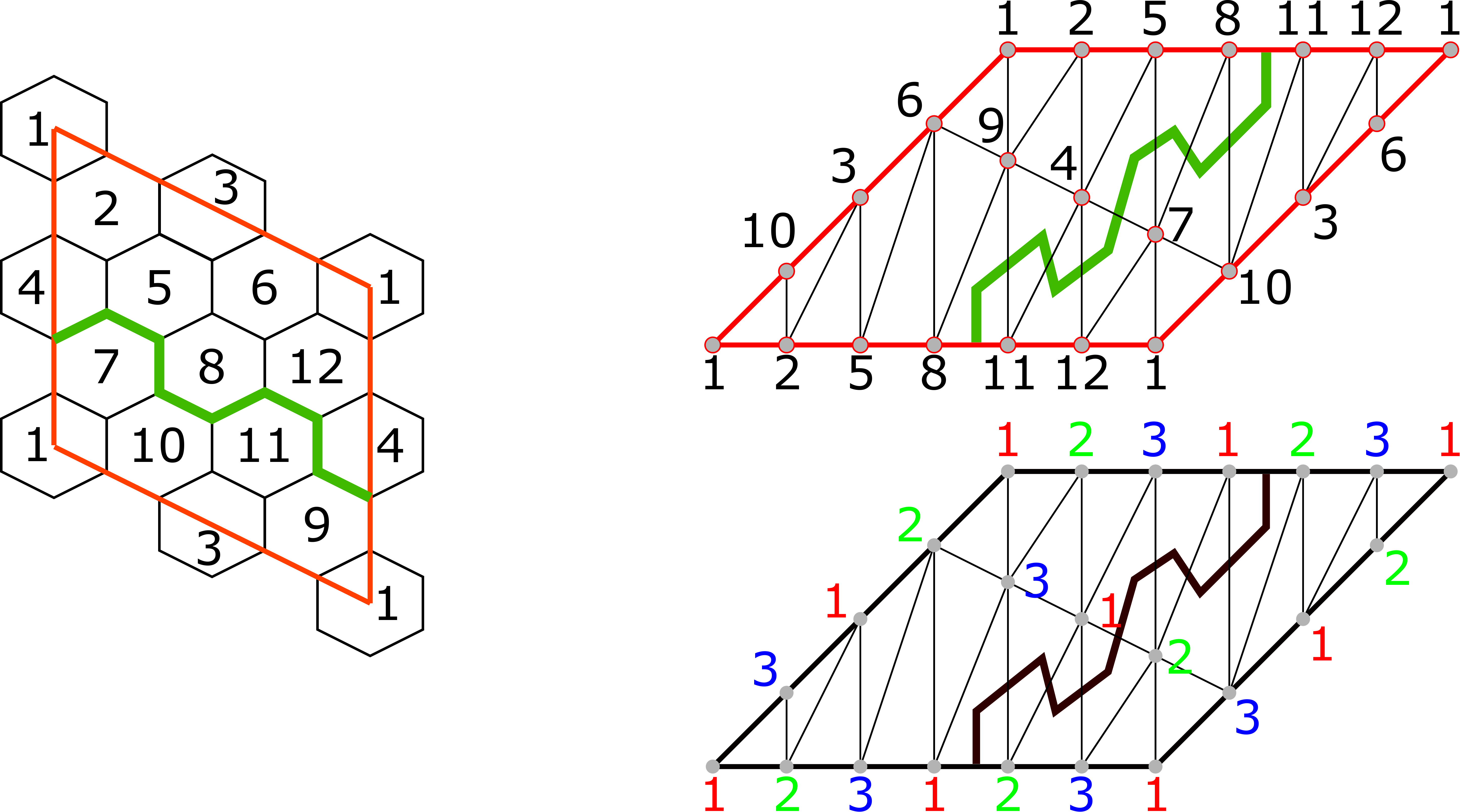}
\end{center}
\caption{Combinatorics of the boundary $\dd Q_3$, and its dual
simplicial complex. The bold line denotes the cycle in $\dd
Q_3\cong T^2$ contractible inside the solid torus $Q_3$. The right
bottom figure shows the proper coloring of vertices of
$\ca{P}_{\St_3}^*$ (i.e. 2-faces of
$Q_3$)}\label{pictTorusAndDual}
\end{figure}

The simplicial complex $\ca{P}_{\St_3}^*$ dual to $Q_3$ is the
triangulation of a $2$-torus with $12$ vertices, shown on
Fig.\ref{pictTorusAndDual} on the right. It can be seen that it
$f$-vector is $(f_{-1}=1,f_0=12,f_1=36,f_2=24)$, $h$-vector is
$(1,9,15,-1)$, and $h'$-vector is $(1,9,15,1)$.

All stabilizers of the action $T^3\curvearrowright
M_{\St_3,\lambda}$ are the coordinate subtori in $T^3$. Therefore,
the characteristic function $\lambda$ of the action is the
chromatic function. It comes from the proper coloring of vertices
of $\ca{P}_{\St_3}^*$ indicated on Fig.\ref{pictTorusAndDual},
bottom right. Propositions \ref{propEquivCohomAlmPoly} and
\ref{propCohomAlmPoly}, applied to $M_{\St_3,\lambda}$ give the
following result.


\begin{thm}\label{thmM3cohomology}
There holds $H^*_T(M_{\St_3,\lambda};R)\cong
R[\ca{P}_{\St_3}^*]\oplus H^*(S^1;R)$. The subring
$A^*(M_{\St_3,\lambda};R)\subset H^*(M_{\St_3,\lambda};R)$
generated by the classes $v_i$ of the characteristic submanifolds
has the form
\[
A^*=A^*(M_{\St_3,\lambda};R)=R[\ca{P}_{\St_3}^*]/\Theta/\ca{I},
\mbox{ where }
\]
\begin{itemize}
\item The ideal $\Theta$ of the Stanley--Reisner ring $R[\ca{P}_{\St_3}^*]$ is
generated by the linear forms $\theta_1=v_1+v_3+v_4+v_8$,
$\theta_2=v_5+v_9+v_{10}+v_{12}$, $\theta_3=v_2+v_6+v_7+v_{11}$.
The numeration of vertices is as shown on
Fig.\ref{pictTorusAndDual}.
\item The ideal $\ca{I}$ is additively generated by the elements
\[
v_8v_{11}+v_7v_8+v_4v_7+v_4v_{11},\quad v_8v_{10}+v_4v_{12},\quad
v_5v_7+v_9v_{11}
\]
(the choice of these elements is noncanonical).
\end{itemize}
Graded components of the subring $A^*$ have dimensions
$(1,0,9,0,12,0,1)$.

The quotient ring $H^*(M_{\St_3,\lambda})/A^+$ has the following
nonempty components: $R$ in degree $0$, $R$ in degree $1$, $R^3$
in degree $2$, and $R$ in degree $5$. The product in the quotient
ring $H^*(M_{\St_3,\lambda})/A^+$ is trivial.

Integral cohomology of $\Xs$ are torsion free, Betti numbers are
$(1,1,12,0,12,1,1)$.
\end{thm}

\begin{proof}
Two statements require an explanation: the form of generators of
the ideal $\ca{I}$, and the torsion freeness of cohomology.
According to Proposition \ref{propCohomAlmPoly}, the relations on
classes $v_I=v_{i_1}\cdots v_{i_k}\in H^*(\Xs)$ are given by all
possible skew forms and all possible cellular cycles in $\dd\Qs$,
which vanish in homology $\Qs$. Each such pair gives a relation
$\sum_{I}\sigma(I)\langle\omega,\lambda_I\rangle v_I$. There is a
unique basis cycle in $\dd\Qs$, which is homologous to zero in
$\Qs$. This cycle can be recognized by analyzing the image of the
moment map (see.\ref{pictNecklace}). The vanishing cycle is shown
on Fig.\ref{pictTorusAndDual}.

Torsion freeness of $\Zo[K]/\Theta/\ca{I}$ can be checked by
direct computation based on part~1 of Proposition
\ref{propMyRelations}.
\end{proof}

%

\section{The twin manifold of $M_{\St_n,\lambda}$}\label{secTwin}

In \cite{ABhess} we introduced a notion of a twin manifold in the
variety of complete complex flags. Given a smooth $T$-invariant
submanifold $X\subset M_{\lambda}\cong \Fl_{n}$, we construct
another smooth $T$-invariant submanifold
$\tilde{X}=p_lp_r^{-1}(X)$, called the \emph{twin} of $X$, where
$p_l\colon U(n)\to\Fl_n$ (resp.$p_r\colon U(n)\to\Fl_n$) is the
quotient map defined by the left free action of a torus on $U(n)$
(resp. right action):
\begin{equation}\label{eqTwoMaps}
\Fl_n\cong T^n\backslash
\raisebox{1pt}{$U(n)$}\stackrel{p_l}{\longleftarrow}U(n)\stackrel{p_r}{\longrightarrow}
\raisebox{1pt}{$U(n)$}/T^n\cong \Fl_n.
\end{equation}

It can be seen that $\tilde{X}/T\cong X/T$ since both spaces are
homeomorphic to the double quotient $T\backslash
p_lp_r^{-1}(X)/T$. However, the characteristic data of the
$T$-manifolds $X$ and $\tilde{X}$ are different in general. The
twins may be nondiffeomorphic.

\begin{con}
Let us describe the twin of the manifold $M_{\St_n,\lambda}$. A
complex flag in $\Co^{n+1}$ can be naturally identified with the
sequence of 1-dimensional linear subspaces
$L_0,L_1,\ldots,L_n\subset \Co^{n+1}$, which are pairwise
orthogonal: $L_i\perp L_j$, $i\neq j$. Consider a diagonalizable
operator $S\colon \Co^{n+1}\to\Co^{n+1}$ with distinct real
eigenvalues, and define the subset $X_n\subset \Fl_{n+1}$:
\[
X_n=\{\{L_i\}\in \Fl_{n+1}\mid SL_i\subset L_0\oplus L_i \mbox{
for }i\neq 0\}.
\]
The action of $T^{n+1}$ on $\Co^{n+1}$ induces an effective action
of $T^n=T^{n+1}/\Delta(T^1)$ on the space $X_n$. We may assume
$S=\Lambda=\diag(\lambda_0,\lambda_1,\ldots,\lambda_n)$.

Note that there is also an action of $\Sigma_n$ on $X_n$ which
permutes the lines $L_1,\ldots,L_n$. This action commutes with the
$T^n$-action, hence there is a combined action of the direct
product $T^n\times\Sigma_n$ on $X_n$.
\end{con}

\begin{prop}
The space $X_n$ is the twin of $M_{\St_n,\lambda}$. In particular,
$X_n$ is a smooth manifold. Its orbit space by the action of $T^n$
is isomorphic to $Q_n$ as a manifold with corners and is homotopy
equivalent to $\Sq_{n-1}$. The orbit space by the action of
$T^n\times\Sigma_n$ is diffeomorphic to the polytope $\B^n$.
\end{prop}

\begin{proof}
We recall the construction of \cite{ABhess}. For a Hermitian
matrix $A$, consider its spectral decomposition $A=U^{-1}\Lambda
U$. Here the unitary operator $U$ is defined up to left
multiplication by diagonal matrices. Given a subspace $X\subset
M_\lambda$ of matrices with the given spectrum, such that $X$ is
preserved by the torus action, we consider the twin space
\[
\tilde{X}=\{A\in M_\lambda\mid A=U\Lambda U^{-1}, \mbox{ where }
U^{-1}\Lambda U\in X\}
\]
Let $e_0,e_1,\ldots,e_n$ be the standard basis $\Co^{n+1}$. Then
$\tilde{X}$ is identified with the collection of flags
\[
U\langle e_0\rangle\subset U\langle e_0,e_1\rangle\subset\cdots
U\langle e_0,e_1,\ldots,e_n\rangle
\]
for all possible unitary matrices $U\in U(n)$ with the property
$U^{-1}\Lambda U\in X$.

Let $L_i=U\langle e_i\rangle$. The condition $U^{-1}\Lambda U\in
M_{\St_n,\lambda}$ is equivalent to
\[
U^{-1}\Lambda U(e_i)\subset \langle e_0,e_i\rangle,\mbox{ for
}i\neq 0,
\]
which is the same condition as
\[
\Lambda L_i\subset L_0\oplus L_i.
\]
Hence the twin of $M_{\St_n,\lambda}$ is exactly $X_n$. Since
$M_{\St_n,\lambda}$ is smooth, so is $X_n$. The orbit spaces of a
manifold and its twin coincide (see \cite{ABhess} for details).
\end{proof}

\begin{rem}
In \cite{ABhess} we noticed that Hessenberg varieties are the
twins of manifolds of staircase isospectral matrices. Note that
unlike this case, the twin $X_n$ of $M_{\St_n,\lambda}$ is not an
algebraic subvariety in $\Fl_{n+1}$.
\end{rem}

\begin{rem}
The manifold $M_{\St_{n},\lambda}$ is a submanifold of
$M_{\St_{n}}\cong \Ro^{3n+1}$ defined by a system of smooth
functions with nondegenerate intersections of level surfaces.
Hence $M_{\St_{n},\lambda}$ has trivial normal bundle, and all its
Pontryagin classes and numbers vanish. However, this may not be
the case for its twin $X_n$. This makes the twin $X_n$ even more
interesting object from topological point of view.
\end{rem}

\begin{con}
Let us describe the characteristic function on $X_n$. Recall that
the facets of the orbit space $X_n/T\cong Q_n$ are encoded by
clusters submitted to the partitions
$\{\{j\},\{0,1,\ldots,\widehat{j},\ldots,n\}\}$, $j\neq 0$, see
Section \ref{secTreeMatricesCombinatorics}. A cluster is given by
\[
\{\{p(j)\},\{p(0),p(1),\ldots,\widehat{p(j)}, \ldots,p(n)\}\}
\]
for a bijection $p\colon \{0,1,\ldots,n\} \to \{0,1,\ldots,n\}$.
The facet $F_{[p]}$ corresponding to this cluster, consists of
flags $\{L_i\}\in X_n$ such that $L_j=e_{p(j)}$. These flags are
stabilized by the circle subgroup $T_{p(j)}\subset T^n$, which is
the image of $p(j)$-th coordinate circle of $T^{n+1}$ in the
quotient $T^n=T^{n+1}/\Delta(T^1)$. In particular, the
characteristic function of $X_n$ takes $n+1$ values. This function
is not chromatic.
\end{con}

\begin{figure}
\begin{center}
\includegraphics[scale=0.3]{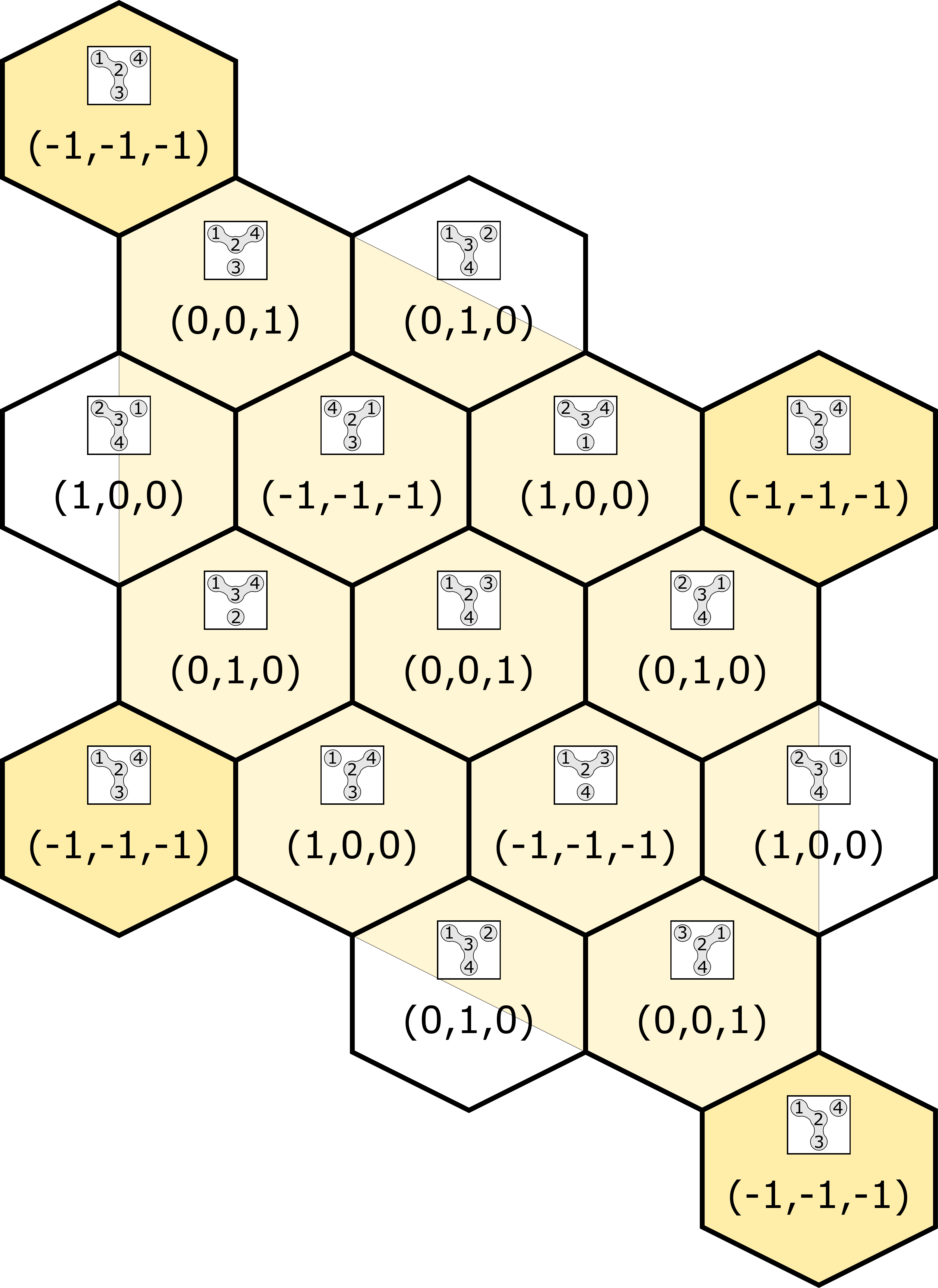}
\end{center}
\caption{The characteristic function for the twin of
$M_{\St_3,\lambda}$}\label{pictHexTorusTwin}
\end{figure}

\begin{ex}
The values of characteristic function on $X_3$ are shown on
Fig.\ref{pictHexTorusTwin}. The characteristic function takes 4
values in $\Zo^3$ which sum to zero, hence we can assume that its
values are $(1,0,0)$, $(0,1,0)$, $(0,0,1)$, and $(-1,-1,-1)$.
These values come from a proper 4-coloring of facets of the
hexagonal subdivision of a torus. The coloring of a facet is
determined by the number standing on a separate vertex of cluster
encoding the facet.

Note that the quotient map $X_3\to X_3/T\cong Q_3$ admits a
section. Indeed, the space $Q_3\cong S^1$ has no second
cohomology, hence every principal $T$-bundle over the interior of
$Q_3$ is trivial. The following proposition is completely similar
to Theorem \ref{thmM3cohomology}.
\end{ex}

\begin{prop}\label{propX3cohomology}
There holds $H^*_T(X_3;R)\cong R[\ca{P}_{\St_3}^*]\oplus
H^*(S^1;R)$. The subring $A^*(X_3;R)\subset H^*(X_3;R)$ generated
by the classes $v_i$ of the characteristic submanifolds has the
form
\[
A^*=A^*(X_3;R)=R[\ca{P}_{\St_3}^*]/\Theta/\ca{I}, \mbox{ where }
\]
\begin{itemize}
\item The ideal $\Theta$ of the Stanley--Reisner ring $R[\ca{P}_{\St_3}^*]$ is
generated by the linear forms
$\theta_1=v_4+v_6+v_{10}-v_1-v_5-v_{11}$,
$\theta_2=v_3+v_7+v_{12}-v_1-v_5-v_{11}$,
$\theta_3=v_2+v_8+v_9-v_1-v_5-v_{11}$ under the numeration of
vertices shown on Fig.\ref{pictTorusAndDual}.
\item The ideal $\ca{I}$ is additively generated by the elements
\[
v_5v_7-v_7v_8+v_8v_{11}-v_{11}v_{12},\quad
v_8v_{10}-v_8v_{11}+v_4v_{11}-v_4v_9,\quad
v_4v_7-v_5v_7+v_{11}v_{12}-v_4v_{11}
\]
(the choice of these elements is noncanonical).
\end{itemize}
Graded components of the subring $A^*$ have dimensions
$(1,0,9,0,12,0,1)$. Integral cohomology of $\Xs$ are torsion free,
Betti numbers are $(1,1,12,0,12,1,1)$.
\end{prop}

The second Pontryagin class of $X_3$ is given by the class
$p_1=\sum_{i=1}^{12}v_i^2\in A^*\subset H^*(X_3)$. It can be shown
that this class is nontrivial. Actually, the integral of this
class over any characteristic submanifold equals $\pm8$ (one needs
to introduce an omniorientation to specify the sign). This
calculation can be done by simplifying the expression $v_ip_1$,
with the use of the relations in the cohomology ring given by
Proposition \ref{propX3cohomology}.

\section*{Acknowledgements}

The authors are deeply grateful to Tadeusz Januszkiewicz, from
whom we knew about the general spaces of sparse isospectral
matrices and, in particular, about the space of isospectral arrow
matrices and its basic properties.

\end{document}